\documentclass[amsa]{ipart}

\RequirePackage{hyperref}

\usepackage{mathrsfs}
\usepackage{bm}
\usepackage{bbm}

\startlocaldefs
\theoremstyle{plain}
\newtheorem{thm}{Theorem}[section]

\newtheorem{lemma}{Lemma}[section]
\newtheorem{corollary}{Corollary}[section]

\theoremstyle{remark}
\newtheorem{definition}{Definition}[section]
\newtheorem{remark}{Remark}[section]

\newcommand{\Pb}{{\mathsf{P}}} 
\newcommand{\Eb}{{\mathsf{E}}} 
\newcommand{\Var}{{\mathsf{Var}}} 
\newcommand{\Dsf}{{\mathsf{D}}}
\newcommand{\Hyp}{{\mathsf{H}}} 
\newcommand{\F}{\mathsf{F}}

\newcommand{\Fc}{{ \mathscr{F}}} 

\newcommand{\Nc}{{ \mathscr{N}}}
\newcommand{\Ni}{\Nc_0\setminus i}

\newcommand{\mrm}[1]{\mathrm{#1}}

\newcommand{\D}{{\mrm{d}}}

\newcommand{\Iin}{{\Theta_{\mrm{in}}}}

\newcommand{\mc}[1]{\mathcal{#1}} 
\newcommand{\cB}{{\mc{B}}}
\newcommand{\cA}{{\mc{A}}}

\newcommand{\cN}{{\mc{N}}}

\newcommand{\mb}[1]{\mathbf{#1}} 
\newcommand{\Xb}{{\mb{X}}}

\newcommand{\ab}{{\mb{a}}}

\newcommand{\Cb}{{\mb{C}}}

\newcommand{\e}{{\mb{e}}}

\newcommand{\mbs}[1]{\bm{#1}} 
\newcommand{\alphab}{{\mbs{\alpha}}}

\newcommand{\mbb}[1]{\mathbb{#1}} 
\def\One{\mathchoice{\rm 1\mskip-4.2mu l}{\rm 1\mskip-4.2mu l}
{\rm 1\mskip-4.6mu l}{\rm 1\mskip-5.2mu l}}
\newcommand\Ind[1]{{\One_{\{#1\}}}} 
\newcommand{\Rbb}{\mbb{R}} 
\newcommand{\Mi}{\mbb{M}} 
\newcommand{\class}{{\mbb{C}}}
\newcommand{\Qbb}{{\mbb{Q}}}

\newcommand{\wtX}{{\widetilde{X}}}

\newcommand{\wtlambda}{{\widetilde{\lambda}}}
\newcommand{\wtS}{{\widetilde{S}}}

\newcommand{\hT}{\widehat{T}}
\newcommand{\hLa}{\widehat{\Lambda}}
\newcommand{\hla}{\widehat{\lambda}}
\newcommand{\oX}{\overline{X}}

\newcommand{\xra}{\xrightarrow} 

\newcommand{\abs}[1]{\left\vert#1\right\vert}
\newcommand{\set}[1]{\left\{#1\right\}}

\newcommand{\brc}[1]{\left(#1\right)}
\newcommand{\brcs}[1]{\left[#1\right]}

\renewcommand{\le}{\leqslant} 
\renewcommand{\ge}{\geqslant}

\newcommand{\ignore}[1]{} 

\endlocaldefs

\pubyear{2024, Special issue: Articles of optimization, wavelets, financial mathematics, and mathematical problems related to statistics, {\bf in memory of Professor Tze Leung Lai}}
\volume{0}
\issue{0}
\firstpage{1}
\lastpage{1}

\begin{document}

\begin{frontmatter}

\title[Nearly Optimum Properties of Multi-Decision Sequential  Rules] {Nearly Optimum Properties of Certain Multi-Decision Sequential Rules for General 
Non-i.i.d.\ Stochastic Models}

\begin{aug}  
\author{\fnms{Alexander G.} \snm{Tartakovsky}\ead[label=e1]{alexg.tartakovsky@gmail.com}}
\address{AGT StatConsult, Los Angeles, CA \\ USA\\ 
\printead{e1}}
\end{aug}
\received{\sday{29} \smonth{4} \syear{2024}}

\begin{abstract}
Dedicated to the memory of Professor Tze Leung Lai, this paper introduces three multi-hypothesis sequential tests. These tests are derived from one-sided versions of the sequential probability ratio test and its modifications. They are proven to be first-order asymptotically optimal for testing simple and parametric composite hypotheses when error probabilities are small. These tests exhibit near optimality properties not only in classical i.i.d. observation models but also in general non-i.i.d. models, provided that the log-likelihood ratios between hypotheses converge $r$-completely to positive and finite numbers. These findings extend the seminal work of Lai (1981) on two hypotheses.
\end{abstract}




\end{frontmatter}


\section{Introduction}\label{s:Intro}

Let $X_1,X_2,\ldots\,$ be independent and identically distributed (i.i.d.) random variables available for observation  and $\{\Pb_\theta, \theta\in \Theta\}$ be a family of distributions 
with densities $p_\theta(x)$ with respect to some non-degenerate, sigma-finite measure, so the joint density of the vector $\Xb^n_1=(X_1,\dots,X_n)$ is 
$
p_\theta(\Xb^n_1) = \prod_{t=1}^n p_\theta(X_t).
$
For testing two simple hypotheses $\Hyp_0: \theta=\theta_0$ and $\Hyp_1: \theta=\theta_1$, Wald~\cite{wald45, wald47} proposed a sequential probability ratio test (SPRT) that is based on comparing the likelihood ratio between these hypotheses, $\Lambda_{\theta_1,\theta_0}(n)=\prod_{t=1}^n [p_{\theta_1}(X_t)/p_{\theta_0}(X_t)]$,
with two thresholds. Wald and Wolfowitz~\cite{wald48} proved that the SPRT has a remarkably strong optimality property -- it minimizes the expected sample size under both 
hypotheses in the class of all tests with error probabilities of Type I and Type II upper-bounded by the given numbers.  
However, the SPRT may perform poorly for parameter values different from the putative values $\theta_0$ and $\theta_1$. 

To address this issue, for testing composite hypotheses  $\Hyp_0: \theta\in \Theta_0$ versus $\Hyp_1: \theta\in \Theta_1$, Wald~\cite{wald47} suggested the 
mixture-likelihood-ratio  approach for modifying the SPRT.  Let $\pi_0(\theta)$ and $\pi_1(\theta)$  be two ``mixing'' probability densities (more generally weights) and let 
\[
\Lambda^\pi(n) = \frac{\int_{\Theta_1}  \prod_{k=1}^n p_\theta(X_k) \pi_1(\theta)\, \D\theta}{\int_{\Theta_0}  \prod_{k=1}^n p_\theta(X_k) \pi_0(\theta) \, \D\theta} 
\]
be the mixture likelihood ratio. Then replacing the likelihood ratio $\Lambda_{\theta_1,\theta_0}(n)$ used in the SPRT by this mixture likelihood ratio $\Lambda^\pi(n)$ 
leads to the mixture SPRT.  Applying Wald's likelihood ratio identity, one can easily obtain the upper bounds on the average error probabilities:
\[
 \int_{\Theta_0} \Pb_\theta(\text{accept} ~ \Hyp_1) \pi_0(\theta) \, \D\theta  ~~ \text{and} ~~   \int_{\Theta_1} \Pb_\theta(\text{accept} ~ \Hyp_0) \pi_1(\theta) \, \D\theta.
\]
However, for practical purposes, it is preferable to bound not only the average error probabilities but also the maximum error probabilities, represented by 
$\sup_{\theta\in\Theta_0}  \Pb_\theta(\text{accept} ~ \Hyp_1)$ and $\sup_{\theta\in\Theta_1}  \Pb_\theta(\text{accept} ~ \Hyp_0)$.
Unfortunately, how to obtain the upper bounds on these maximal error probabilities for the mixture SPRT is generally unclear.

An alternative approach is the generalized likelihood ratio (GLR) method from the classical Neyman--Pearson fixed-sample-size theory. This approach replaces the likelihood ratio 
$\Lambda_{\theta_1,\theta_0}(n)$ used in the SPRT by the GLR statistics
\begin{equation}\label{GLRstatiid}
\widehat\Lambda_{i}(n) = \frac{\sup_{\theta \in \Theta} \prod_{t=1}^n p_\theta(X_t)}{\sup_{\theta \in \Theta_i} \prod_{t=1}^n p_\theta(X_t)} = 
 \frac{\prod_{t=1}^n p_{\hat\theta_n}(X_t)}{\sup_{\theta \in \Theta_i} \prod_{t=1}^n p_\theta(X_t)}, ~~ i=0,1,
\end{equation}
 where $\hat\theta_n=\arg\sup_{\theta \in \Theta} \prod_{t=1}^n p_\theta(X_t)$ is the maximum likelihood estimator of $\theta$. 
 Obtaining upper bounds for maximal error probabilities is also a challenge since the GLR statistics are not viable likelihood ratios anymore and as a result, 
 Wald's likelihood ratio identity cannot be used for this purpose.
 The GLR method has been developed in numerous publications, encompassing both Bayesian and frequentist settings, particularly when the cost of observations or 
 the error probabilities are small. See, e.g, Schwarz~\cite{Schwarz-AMS1962}, Wong \cite{WongAMS-1968},  
Lorden~\cite{Lorden1,lorden-as73,Lorden-unpublished-1977}, Lai~\cite{Lai-AS1988}, Lai and Zhang~\cite{LaiZhang-SQA1994}, Chan and Lai~\cite{ChanLai-AS2000}.

The other way is to employ a combination of mixture-based and GLR approaches exploiting the statistics  
\begin{equation}\label{MixGLRstatiid}
\widehat\Lambda_{i}^\pi(n) = \frac{\int_{\Theta} \prod_{t=1}^n p_\theta(X_t) \pi(\theta) \, \D\theta}{\sup_{\theta \in \Theta_i} \prod_{t=1}^n p_\theta(X_t)} , ~~ i=0,1,
\end{equation}
where $\Theta=\Theta_0\cup \Theta_1$. This modification of the SPRT has the advantage that upper bounds on the maximal error probabilities $\sup_{\theta\in\Theta_0}  \Pb_\theta(\text{accept} ~ \Hyp_1)$ 
and $\sup_{\theta\in\Theta_1}  \Pb_\theta(\text{accept} ~ \Hyp_0)$ can be obtained in the same way as for the SPRT (see Lemma~\ref{Lem:PEMMSPRT} 
in Section~\ref{ss:PEMMSPRT} for the more general multi-hypothesis and non-i.i.d.\ case).
 
 Note that the GLR method is adaptive. Yet another adaptive approach is to replace the ``global'' maximum likelihood estimator in the GLR statistic \eqref{GLRstatiid} by  
 one-step delayed estimators (at each step), that is, instead of $\widehat{\Lambda}_{i}(n)$ to use the adaptive likelihood ratio statistics
 \begin{equation*}
\widehat\Lambda_{i}^*(n) = \frac{\prod_{t=1}^n p_{\hat\theta_{t-1}}(X_t)}{\sup_{\theta \in \Theta_i} \prod_{t=1}^n p_\theta(X_t)}, ~~ i=0,1.
\end{equation*}
In this case, the Wald likelihood ratio identity can still be applied to upper-bound the error probabilities of the corresponding adaptive SPRT.  
The adaptive SPRT is therefore a very attractive alternative to the GLR-SPRT and 
the mixture-based SPRT.  The idea of this test goes back to the works by Robbins and Siegmund \cite{RobbinsSiegmund-Berkeley70,RobbinsSiegmund-AS74} 
who were the first who suggested a one-sided adaptive test in the context of power $1$ tests. 
 
 In the present paper, we develop the asymptotic theory of multi-hypothesis sequential testing for general stochastic models not restricted to the i.i.d.\ assumption.
 Specifically, we consider models with dependent and non-identically distributed data of a very general structure with minimal assumptions related 
 to the $r$-complete convergence of normalized log-likelihood 
 ratio processes to certain positive numbers.  The approach modifies and extends the ideas in Lai's seminal work~\cite{Lai-as81-SPRT} for 
 two simple hypotheses, and Tartakovsky's contributions~\cite{TartakovskySISP98} for multiple simple hypotheses.

\section{Problem formulation} \label{s:PF}

We are interested in the following general discrete-time multi-hypothesis model with parametric composite hypotheses.  Let $(\Omega, \Fc ,\Fc_n, \Pb_\theta)$, $n \ge 1$ be a filtered 
probability space with standard assumptions about monotonicity of the $\sigma$-algebras~$\Fc_n$. The parameter $\theta=(\theta_1,\dots,\theta_l)$ belongs to a subset~$\Theta$ 
of the $l$-dimensional Euclidean space~$\Rbb^l$. Random variables $X_1, X_2, \dots$ are observed sequentially taking values in a measurable space $(\Omega, \Fc)$.
The sub-$\sigma$-algebra $\Fc_n =  \sigma (X_1,\dots,X_n)$ of~$\Fc$ is generated by the sequence 
$\{X_t\}_{t\ge 1}$  observed up to time~$n$. The hypotheses to be tested are $\Hyp_i:~ \theta \in \Theta_i$, $i = 0, 1,\dots,N$ ($N \ge 1$), 
where~$\Theta_i$ are disjoint subsets of~$\Theta$. We also suppose that there is an indifference zone $\Iin\subset \Theta$ in which there are no restrictions on the error probabilities. 
It is assumed that the subsets $\Iin$ and $\Theta_i$ ($i = 0, 1,\dots,N$) are disjoint. The indifference zone, where any decision is acceptable, is usually introduced because the 
correct action is not critical and often not even possible when the hypotheses are too close, which is perhaps the case in most practical applications. 
However, if needed $\Iin$ may be an empty set. The probability measures~$\Pb_\theta$ and~$\Pb_{\vartheta}$ are assumed to be locally mutually absolutely continuous, i.e., 
the restrictions $\Pb_\theta^n$ and $\Pb_{\vartheta}^n$ of these measures to the sub-$\sigma$-algebras~$\Fc_n$ are equivalent for all $1 \le n < \infty$ and 
all distinct values $\theta, \vartheta\in \Theta$.

Write $\Xb^n_1=(X_1,\dots,X_n)$ for the sample of size $n$. It is convenient to represent the general probabilistic model in terms of densities as
\begin{equation}\label{jointdenscomp}
p_{\theta, n}(\Xb^n_1) = \prod_{t=1}^n f_{\theta,t}(X_t| \Xb_1^{t-1}), \quad \theta \in \Theta,
\end{equation}
where $p_{\theta,n}(\Xb^n_1)$ is the joint density of the sample $\Xb^n_1$ (i.e., density of $\Pb_\theta^n$ with respect to some sigma-finite measure) and $f_{\theta,t}(X_t| \Xb^{t-1}_1)$ 
is the corresponding conditional density.

For the sake of brevity in what follows we will often write $\Nc_0$ for the set $\{0,1,\dots,N\}$. 

 A multi-hypothesis sequential test~$\Dsf=(T, d)$, where~$T$ is a stopping time with respect to the filtration $\{\Fc_n \}_{n \ge 1}$, 
and $d=d_T(\Xb_1^{T})\in \{0, 1,\dots,N\}$ is an $\Fc_T$-measurable terminal decision rule specifying which hypothesis is to be accepted once 
observations have stopped. The hypothesis~$\Hyp_i$ is accepted if~$d=i$ and rejected if~$d \neq i$, i.e., $\set{d=i}=\set{T<\infty, ~  ~ \text{accepts $H_i$}}$, $i \in \Nc_0$. 

The quality of a sequential test is judged based on its error probabilities and expected sample sizes or more generally on the moments of the sample size. 
Let $\Eb_\theta$ denote expectation under $\Pb_\theta$, $\theta\in \Theta$ and let 
\[
\alpha_{ij}(\Dsf, \theta)= \Eb_\theta\brcs{\Ind{d=j}}= \Pb_\theta(d=j), \quad \theta\in \Theta_i,  \quad i \neq j \quad (i,j=0,1,\dots,N) 
\]
denote the probability that the test~$\Dsf$ accepts the hypothesis~$\Hyp_j$ when the true value of the parameter~$\theta$ is fixed within the subset~$\Theta_i$. Hereafter 
$\Ind{\cA}$ denotes an indicator of a set $\cA$. By 
\begin{equation}\label{Mclassescomp}
\begin{aligned}
\class(\alphab) & = \set{\Dsf: \sup_{\theta \in \Theta_i} \alpha_{ij}(\Dsf, \theta)  \le \alpha_{ij} ~ \text{for} ~ i, j = 0,1,\dots,N, \, i \neq j}
\end{aligned}
\end{equation}
denote the class of tests with the maximal error probabilities that do not exceed the predefined values~$\alpha_{ij}\in (0,1)$, where $\alphab=(\alpha_{ij})_{i,j\in \Nc_0, i \neq j}$ 
is the matrix of given constraints. 

Note that the class $\class(\alphab)$ confines the error probabilities in the regions~$\Theta_i$ but not in the indifference zone~$\Iin$ where the hypotheses are too close to be distinguished with 
the given and relatively low error probabilities.  However, ideally, we would like to minimize the expected sample size  $\Eb_\theta [T]$ for all possible parameter values, including those in 
the indifference zone. Unfortunately, there is no such test, since the structure of the test that minimizes the expected sample size $\Eb_\theta [T]$ at a specific parameter 
value~$\theta=\tilde\theta$  depends on~$\tilde\theta$. 
However, this problem may be solved asymptotically when the error probabilities are small. That is, we are interested in finding multi-hypothesis tests $\Dsf=(T,d)$ that minimize the 
expected sample size
$\Eb_\theta [T]$ uniformly for all $\theta \in \Theta$ in the class of tests~$\class(\alphab)$  as 
$\alpha_{ij}$ approach zero. More generally, we are interested in finding the tests that minimize asymptotically the moments of the stopping time distribution up to some order $r \ge 1$:
\begin{equation}\label{OptuniformProblem}
\inf_{\Dsf\in\class(\alphab)} \Eb_\theta [T^r] \quad \text{as}~ \alpha_{\max} \to 0  \quad \text{uniformly in}~~ \theta\in \Theta,
\end{equation}
where $\alpha_{\max}=\max_{i,j\in \Nc_0, i \neq j} \alpha_{ij}$ and $\Theta = \sum_{i=0}^N \Theta_i  + \Iin$. Specifically, we need to construct a sequential test 
$\Dsf_*=(T_*, d_*)$ that is first-order asymptotically optimal under certain quite general conditions, i.e., 
\begin{equation}\label{FOASoptimcomp}
\frac{\inf_{\Dsf\in\class(\alphab)} \Eb_\theta [T^r]}{\Eb_\theta [T_*^r]} = 1+o(1)  \quad \text{as}~ \alpha_{\max} \to 0 \quad \text{for all}~ \theta\in\Theta.
\end{equation}
Hereafter $o(1) \to 0$.

We will also consider the case of simple hypotheses where the parameter $\theta$ takes $N+1$ values $\theta_0, \theta_1, \dots, \theta_N$ or, more generally, 
the distributions under hypotheses  $\Hyp_i$ have joint distinct densities
\begin{equation}\label{jointdenssimple}
p(\Xb^n_1 | \Hyp_i)= p_{i,n}(\Xb^n_1) = \prod_{t=1}^n f_{i,t}(X_t| \Xb^{t-1}_1), \quad i=0,1,\dots,N,
\end{equation}
where $f_{i,t}(X_t| \Xb^{t-1}_1)$ are the corresponding conditional densities. We need to construct a sequential test $\Dsf_*=(T_*, d_*)$ such that 
\begin{equation}\label{FOASoptimsimple}
\frac{\inf_{\Dsf\in\class_{\rm sim}(\alphab)} \Eb_i [T^r]}{\Eb_i [T_*^r]} = 1+o(1)  \quad \text{as}~ \alpha_{\max} \to 0 \quad \text{for all}~ i=0,1,\dots,N,
\end{equation}
where $\Eb_i$ is expectation under measure $\Pb_i$ and
\begin{equation}\label{Mclassessimple}
\class_{\rm sim}(\alphab)  = \set{\Dsf: \Pb_i(d=j)  \le \alpha_{ij} ~ \text{for} ~ i, j = 0,1,\dots,N, \, i \neq j}.
\end{equation}

\begin{remark}\label{Rem:Asymcase}
In what follows, for practical purposes, we will suppose that the ratios of logarithms of error probabilities $\log \alpha_{ij}/\log \alpha_{ks}$ are bounded away from zero and infinity,
i.e., the error probabilities approach zero in such a way that for all $i,j$
\begin{equation}\label{Asymcase}
\frac{|\log \alpha_{ij}|}{|\log \alpha_{\max}|} \sim c_{ij} ~~ \text{as}~ \alpha_{\max} \to 0, ~~ 0 < c_{ij} \le 1  .
\end{equation}
 This assumption guarantees that any error probability $\alpha_{ij}$ does not go to zero at an exponentially faster (or slower)
rate than any other $\alpha_{ks}$. However, we do not require (as it often happens in the literature) that the $\alpha_{ij}$'s go to zero at the same rate, in which case 
all $c_{ij}$ are equal to 1.  In other words, we consider an asymptotically asymmetric case rather than an asymptotically symmetric one when $c_{ij}=1$. The reason is that
there are many important problems for which the error probabilities may be orders of magnitude different. A typical example is the problem of target detection when the
probabilities of a false alarm are required to be substantially smaller than the probabilities of target missing (miss detection).
\end{remark}

\section{Asymptotic lower bounds for performance metrics} \label{s:ALB}

For $n \ge 1$ and $\theta, \vartheta \in\Theta$ define the likelihood ratio (LR) and the log-likelihood ratio (LLR) for the sample $\Xb^n_1$ 
between the distinct points $\theta$ and $\vartheta$
\begin{equation}\label{LRcomposite}
\begin{aligned}
\Lambda_{\theta, \vartheta}(n) & = \frac{p_{\theta,n}(\Xb^n_1)}{p_{\vartheta,n}(\Xb^n_1)} = \prod_{t=1}^n \frac{f_{\theta,t}(X_t | \Xb^{t-1}_1)}{f_{\vartheta,t}(X_t | \Xb^{t-1}_1)},
\\
\lambda_{\theta, \vartheta}(n) & = \log \Lambda_{\theta, \vartheta}(n) =  \sum_{t=1}^n \log \brcs{\frac{f_{\theta,t}(X_t | \Xb^{t-1}_1)}{f_{\vartheta,t}(X_t | \Xb^{t-1}_1)}}.
\end{aligned}
\end{equation}

Below we show that the lower bounds in class $\class(\alphab)$ for the performance metrics -- the moments of the stopping time $\Eb_{\theta}[T^r]$ -- can be established as long as
the LLR $\lambda_{\theta, \vartheta}(n)$ obeys the strong law of large numbers (SLLN) with a certain rate $\psi(n)$. 

Throughout the paper, we assume that $\psi(t)$ is an increasing one-to-one $\Rbb_+\to \Rbb_+$ function. By $\Psi$ we denote its inverse $\psi^{-1}$ and we assume that both 
$\psi$ and $\Psi$ go to infinity as $t\to\infty$, $\lim_{t\to\infty} \psi(t)=\lim_{t\to\infty} \Psi(t)=\infty$. In addition, we assume the following conditions on the inverse function 
\begin{equation}\label{CondPsi}
\lim_{\delta\to 1} \lim_{t\to\infty} \frac{\Psi(\delta \,  t)}{\Psi(t)} =1
\end{equation}
and
\begin{equation}\label{CondPsi1}
\lim_{t\to\infty} t^{-1} \log \Psi(t)=0.
\end{equation}
Note that conditions \eqref{CondPsi}-\eqref{CondPsi1} are satisfied for the power function $\psi(t)= t^\beta$ with $\beta>0$, but not for the logarithmic function $\psi(t)=\log t$, i.e.,  
the function $\psi(t)$ has to increase not too slowly -- faster than the logarithmic function, $\psi(t)/\log t\to \infty$ as $t\to\infty$.

To obtain asymptotic lower bounds we impose the following condition.
\vspace{2mm}

\noindent $\Cb 1$. {\em Right-tail Condition}. Assume that there are a positive increasing function $\psi(t)$, $\psi(\infty)=\infty$, satisfying condition \eqref{CondPsi}, and a positive continuous 
function $I(\theta,\vartheta)$ with 
\begin{equation}\label{Ipositive} 
\begin{aligned}
\mathop{\min_{j \in \Ni}} \inf_{\vartheta\in\Theta_j}I(\theta,\vartheta) &  >0 \quad \text{for all}~ \theta\in\Theta_i ~ \text{and}~ i\in \Nc_0,
\\
\min_{i \in \Nc_0} \inf_{\vartheta\in\Theta_i}I(\theta,\vartheta) & > 0 \quad \text{for all}~ \theta\in \Theta_{\rm in},
\end{aligned}
\end{equation}
such that for any $\varepsilon>0$ and all $\theta, \vartheta \in \Theta$, $\theta\neq \vartheta$ 
\begin{equation}\label{RTCond}
\lim_{L\to\infty}\Pb_{\theta}\set{\frac{1}{\psi(L)}\max_{1 \le n \le L} \lambda_{\theta, \vartheta}(n)  >(1+\varepsilon) I(\theta, \vartheta)} =0  .
\end{equation}

Remark~\ref{Rem:AS} below shows that $I(\theta,\vartheta)$ can be interpreted as an information ``distance'' between the parameter points $\theta$ and $\vartheta$ and
$\inf_{\vartheta\in\Theta_i} I(\theta,\vartheta) =I(\theta,\Theta_i)$ as a ``distance'' between the parameter value $\theta$ and the subset $\Theta_i$. Thus, conditions 
\eqref{Ipositive} represent separability restrictions between subsets $\Theta_i$, $\Theta_j$ and $\Theta_{\rm in}$ ($i\neq j$). It is intuitively obvious that if the corresponding
distances are zero, then the hypotheses become indistinguishable. 

Additional notation:  $I_i(\theta)= I(\theta,\Theta_i)= \inf_{\vartheta\in\Theta_i} I(\theta,\vartheta)$,
\begin{equation}\label{Fdefinition}
F_{i,\theta}(\varepsilon, \alphab) = \Psi\brc{(1-\varepsilon) \max_{j \in \Ni} \frac{|\log \alpha_{ji}|}{I_j(\theta)} },
\end{equation}
so that 
\begin{equation}\label{Fdefinition2}
F_{i,\theta}(0, \alphab) = \Psi\brc{\max_{j \in \Ni} \frac{|\log \alpha_{ji}|}{I_j(\theta)} }.
\end{equation}
In the following, we will omit $0$ in \eqref{Fdefinition2} and will write $F_{i,\theta}(\alphab)$ for $F_{i,\theta}(0,\alphab)$. 

The following theorem establishes lower bounds for arbitrary moments of the stopping time distribution in the class of sequential or non-sequential multi-hypothesis tests 
$\class(\alphab)$ defined in \eqref{Mclassescomp}.  
In the sequel, we will often write $\theta_i$ for $\theta$ when $\theta$ belongs to the subset $\Theta_i$.

\begin{thm} \label{Th:LBcomposite} 
If there exist an increasing positive function $\psi(t)$, $\psi(\infty)=\infty$,  satisfying condition \eqref{CondPsi}, and a positive function $I(\theta,\vartheta)$, satisfying \eqref{Ipositive}, 
such that right-tail conditions \eqref{RTCond} hold, 
then for every $0 < \varepsilon < 1$
\begin{equation}\label{Prob1composite}
\lim_{\alpha_{\max}\to 0} \inf_{\Dsf\in\class(\alphab)}  \Pb_\theta\set{T >  F_{i,\theta}(\varepsilon, \alphab)} =1 ~~ \text{for all}~ \theta\in\Theta_i ~ \text{and}~ i \in \Nc_0
\end{equation}
and for every  $0 < \varepsilon < 1$
\begin{equation}\label{Prob1composite2}
\lim_{\alpha_{\max}\to 0} \inf_{\Dsf\in\class(\alphab)}  \Pb_\theta\set{T >  \min_{0 \le i \le N} F_{i,\theta}(\varepsilon, \alphab)} =1 ~~ \text{for all}~ \theta\in \Theta_{\rm in}.
\end{equation}
Therefore,  for all $r \ge 1$ as $\alpha_{\max} \to 0$
\begin{align}
\inf_{\Dsf\in \class(\alphab)} \Eb_{\theta}[T^r]  & \ge \brcs{F_{i, \theta}(\alphab)}^r (1+o(1))  ~~  \text{for all}~ \theta\in\Theta_i ~ \text{and}  ~i \in \Nc_0;
\label{LBmomentscomposite1}
\\
\inf_{\Dsf\in \class(\alphab)} \Eb_{\theta}[T^r]  & \ge  \brcs{\min_{i \in \Nc_0} F_{i,\theta}(\alphab)}^r (1+o(1)) \quad \text{for all}~~ \theta\in\Theta_{\rm in} .
\label{LBmomentscomposite2}
\end{align}
\end{thm}
 
 \begin{proof}
 Let $\Dsf=(T,d)$ be an arbitrary test from class~$\class(\alphab)$. It suffices to consider tests that terminate almost surely,
$\Pb_{\theta}(T<\infty)=1$, since otherwise $\Eb_\theta[T^r] = \infty$ and the statement follows trivially.  Changing the measure $\Pb_\vartheta \to \Pb_\theta$ and using Wald's likelihood ratio identity, 
we obtain that for any $s\ge 1$, $C>0$, and any two distinct points $\theta$ and~$\vartheta$
\begin{equation}\label{EqIeqfundamcomp}
\begin{aligned}
&\Pb_{\vartheta}(d=i) =\Eb_{\theta}\bigl\{\Ind{d=i}\Lambda_{\theta,\vartheta}(T)^{-1}\bigr\}
\\
& \ge \Eb_{\theta}\bigl\{\Ind{d=i,T \le s, \Lambda_{\theta,\vartheta}(T)<e^C}\Lambda_{\theta,\vartheta}(T)^{-1}\bigr\}
\\
& \ge e^{-C} \Pb_{\theta}\Bigl(d=i,T \le s, \max_{1 \le n \le s}\Lambda_{\theta,\vartheta}(n)< e^C\Bigr)
\\
& \ge e^{-C} \Bigl\{\Pb_{\theta}(d=i,T \le s)-\Pb_{\theta} \Bigl(\max_{1\le n\le s}\lambda_{\theta,\vartheta}(n)\ge C \Bigr)\Bigr\}.
\end{aligned}
\end{equation}
The last inequality follows from the Boole inequality $\Pb(\cA\cap \cB) \ge \Pb(\cA)-\Pb(\cB^c)$,
where $\cB^c$ is a complement of the event~$\cB$, if we set $\cA=\set{d=i,T\le s}$ and $\cB=\{\max_{1\le n\le s}\lambda_{\theta,\vartheta}(n)<C\}$.
It follows that
\begin{equation} \label{inprtau1}
\Pb_{\theta}(d=i,T\le s) \le \Pb_{\vartheta}(d=i) \, e^C +\Pb_{\theta}\set{\max_{1 \le n \le s}\lambda_{\theta,\vartheta}(n)\ge C},
\end{equation}
and since, by Boole's inequality, 
\[
\Pb_\theta(d=i, T \le s) \ge \Pb_\theta( T\le s) - \Pb_\theta(d \neq i),
\]
we obtain
\begin{equation} \label{inprtau2}
\Pb_{\theta}(T\le s) \le \Pb_\theta(d \neq i) + \Pb_{\vartheta}(d=i) \, e^C +\Pb_{\theta}\set{\max_{1 \le n \le s}\lambda_{\theta,\vartheta}(n)\ge C}.
\end{equation}

Let $\theta = \theta_i \in\Theta_i$ and $\vartheta\notin\Theta_i$ and set 
\begin{align*}
s & =s_{ij}(\varepsilon,\theta_i, \alpha_{ji})=\Psi\brc{ (1-\varepsilon)\frac{|\log \alpha_{ji}|}{I(\theta_i,\vartheta)}}, 
\\
C & = C_{ij}(\varepsilon, \alpha_{ji})= (1+\varepsilon) I(\theta_i,\vartheta) \psi(s_{ij}) =(1-\varepsilon^2) |\log \alpha_{ji}| .
\end{align*}
Using~\eqref{inprtau2} we obtain that for all $j \in \Nc_0 \setminus  i$, $\theta_i\in \Theta_i$ and $i =0,1,\dots,N$
\begin{align*}
 &\Pb_{\theta_i}\set{T \le s_{ij}(\varepsilon,\theta_i, \alpha_{ji})} \le \sum_{k \in \Nc_0 \setminus i} \alpha_{ik} +\alpha_{ji}^{\varepsilon^2}
 \\
 & \quad + \Pb_{\theta_i}\set{\max_{1 \le n\le s_{ij}}\lambda_{\theta_i,\vartheta}(n)\ge(1+\varepsilon)I(\theta_i,\vartheta) \psi(s_{ij})}.
\end{align*}
Since the right-hand side does not depend on~$\Dsf$, we have
\begin{align*}
& \inf_{ \Dsf\in\class(\alphab)}\Pb_{\theta_i}\set{T >  s_{ij}(\varepsilon,\theta_i, \alpha_{ji})} \ge 1-\sum_{k \in \Nc_0\setminus i} \alpha_{ik} - \alpha_{ji}^{\varepsilon^2}
 \\
  &\quad - \Pb_{\theta_i}\set{\frac{1}{\psi(s_{ij})}\max_{1 \le n\le s_{ij}}\lambda_{\theta_i,\vartheta}(n)\ge(1+\varepsilon)I(\theta_i,\vartheta)}.
 \end{align*}
The second and third terms on the right-hand side in the above inequality go to~$0$ as $\alpha_{\max}  \to 0$. The fourth term also goes to $0$ 
for all $0<\varepsilon< 1$ by condition \eqref{RTCond}. Hence, for all $\theta_i\in\Theta_i$, $\vartheta\notin\Theta_i$ and all $j \in \Ni$ 
$$
\inf_{ \Dsf\in\class(\alphab)}\Pb_{\theta_i}\set{T > \Psi\brc{ (1-\varepsilon)\frac{|\log \alpha_{ji}|}{I(\theta_i,\vartheta)}}} \xra[\alpha_{\max} \to 0]{} 1,
$$
which implies \eqref{Prob1composite}.

Next, we prove~\eqref{Prob1composite2} for the indifference zone, $\theta\in\Iin$.  For any $\vartheta=\theta_j\in \Theta_j$, let
\[
K_\theta(\alphab) = \min_{i\in \Nc_0} \max_{j \in \Nc_0\setminus i} \frac{|\log \alpha_{ji}|}{I(\theta, \theta_j)}.
\] 

By~\eqref{EqIeqfundamcomp}, for every $\theta\in\Iin$ and $\theta_j\in\Theta_j$,
\begin{equation*}
\Pb_{\theta}(d=i, T\le s) \le \Pb_{\theta_j}(d=i) e^C  +\Pb_{\theta}\Bigl\{\max_{1 \le n \le s}\lambda_{\theta,\theta_j}(n) \ge C\Bigr\}.
\end{equation*}
Setting $s=s_\theta(\varepsilon,\alphab)=\Psi((1-\varepsilon)K_\theta(\alphab))$ and 
\[
C= (1+\varepsilon) I(\theta,\theta_j) \psi(s_\theta(\varepsilon,\alphab))=(1-\varepsilon^2) I(\theta,\theta_j) K_\theta(\alphab),
\]
we obtain that for all $j \in \Nc_0\setminus i$ and $i=0,1,\dots,N$
\begin{align*}
 & \Pb_{\theta}\set{d=i, T \le s_\theta(\varepsilon,\alphab)} \le \alpha_{ji}^{\varepsilon^2}
 \\
 & \quad 
 + \Pb_{\theta}\set{\max_{1 \le n\le s_\theta(\varepsilon,\alphab)} \lambda_{\theta,\theta_j}(n)\ge (1+\varepsilon) I(\theta,\theta_j) \psi(s_\theta(\varepsilon,\alphab))},
\end{align*}
so that for all $i=0,1,\dots,N$
\begin{equation*}
 \Pb_{\theta}\set{d=i, T \le s_\theta(\varepsilon,\alphab)} \le \beta_{i}^{\varepsilon^2} + \gamma_i(\theta,\alphab,\varepsilon) ,
\end{equation*}
where $\beta_i=\max_{j \in \Nc_0\setminus i} \alpha_{ji}$ and 
\[
\gamma_i(\theta, \alphab, \varepsilon) = \max_{j \in \Nc_0\setminus i} \Pb_{\theta}\set{\frac{1} {\psi(s_\theta(\varepsilon,\alphab))} \max_{1 \le n\le s_\theta(\varepsilon,\alphab)}
\lambda_{\theta,\theta_j}(n) \ge (1+\varepsilon) I(\theta,\theta_j)}  .
\]
Consequently, 
\begin{equation*}
 \Pb_{\theta}\set{T \le s_\theta(\varepsilon,\alphab)} \le \sum_{i=0}^N \brcs{\beta_{i}^{\varepsilon^2} + \gamma_i(\theta,\alphab,\varepsilon)} ,
\end{equation*}
and since the right-hand side does not depend on any test $\Dsf$, we obtain the inequality
\[
 \sup_{\Dsf\in\class(\alphab)}  \Pb_{\theta}\set{T \le \Psi\brc{(1-\varepsilon) \min_{i\in \Nc_0} \max_{j \in \Nc_0\setminus i} \frac{|\log \alpha_{ji}|}{I(\theta, \theta_j)}}} \le 
 \sum_{i=0}^N \brcs{\beta_{i}^{\varepsilon^2} + \gamma_i(\theta,\alphab,\varepsilon))} ,
\]
where by condition \eqref{RTCond}  $\gamma_i(\theta,\alphab,\varepsilon) \to 0$ as $\alpha_{\max}\to 0$ for all $\theta\in \Theta_{\rm in}$ and $\varepsilon \in (0,1)$. It follows that for every $0<\varepsilon <1$ and $\theta\in\Iin$,
\[
\sup_{\Dsf\in\class(\alphab)} \Pb_{\theta}\set{T \le \Psi\brc{(1-\varepsilon) \min_{i\in \Nc_0} \max_{j \in \Nc_0\setminus i} \frac{|\log \alpha_{ji}|}{I_j(\theta)}}}\to 0 
\quad \text{as}~~ \alpha_{\max} \to 0 ,
\]
which completes the proof of~\eqref{Prob1composite2}.

The asymptotic lower bounds~\eqref{LBmomentscomposite1} and \eqref{LBmomentscomposite2} now follow from Chebyshev's inequality. Indeed, by the Chebyshev inequality,
for any $r\ge 1$ and all $\theta\in\Theta_i$ and $i =0,1,\dots,N$,
\[
\inf_{\Dsf\in\class(\alphab)}\Eb_\theta \brcs{\brc{\frac{T}{F_{i,\theta}(\alphab)}}^r} \ge 
\brc{\frac{F_{i,\theta}(\varepsilon, \alphab)}{F_{i,\theta}(\alphab)}}^r \inf_{\Dsf\in\class(\alphab)} \Pb_\theta\set{T>F_{i,\theta}(\varepsilon, \alphab)} .
\]
Since $\varepsilon$ can be arbitrarily small and noting that, by condition \eqref{CondPsi},
\begin{equation}\label{limsupFoverF}
\lim_{\varepsilon\to0} \lim_{\alpha_{\max}\to0}\frac{F_{i,\theta}(\varepsilon, \alphab)}{F_{i,\theta}(\alphab)} =1
\end{equation}
and that by \eqref{Prob1composite}
\[
\inf_{\Dsf\in\class(\alphab)} \Pb_\theta\set{T>F_{i,\theta}(\varepsilon, \alphab)} \xra[\alpha_{\max} \to 0]{} 1,
\]
taking the limit $\varepsilon \to0$, we obtain that  for all $r\ge 1$,  $\theta\in\Theta_i$ and $i =0,1,\dots,N$,
\[
\liminf_{\alpha_{\max}\to 0}  \inf_{\Dsf\in\class(\alphab)}\Eb_\theta \brcs{\brc{\frac{T}{F_{i,\theta}(\alphab)}}^r} \ge 1,
\]
which completes the proof of assertion \eqref{LBmomentscomposite1}. 

Analogously, define
\[
\tau_\theta(\varepsilon,\alphab) = \frac{T}{\min_{i \in \Nc_0}F_{i,\theta}(\varepsilon, \alphab) }.
\]
By the Chebyshev inequality, for  any $r\ge 1$ and all $\theta\in\Theta_{\rm in}$
\[
\inf_{\Dsf\in\class(\alphab)}\Eb_\theta [\tau_\theta(0,\alphab)^r ] \ge \brc{\frac{\min_{i \in \Nc_0} F_{i,\theta}(\varepsilon, \alphab)}{\min_{i \in \Nc_0}F_{i,\theta}(\alphab)}}^r 
\inf_{\Dsf\in\class(\alphab)} \Pb_\theta\set{\tau_\theta(\varepsilon,\alphab) > 1},
\]
where by \eqref{Prob1composite2}
\[
\inf_{\Dsf\in\class(\alphab)} \Pb_\theta\set{\tau_\theta(\varepsilon,\alphab) > 1} \xra[\alpha_{\max} \to 0]{} 1 \quad \text{for every}~ 0<\varepsilon <1.
\]
So taking the limit $\varepsilon \to 0$ and accounting for \eqref{limsupFoverF}, we obtain  that
 for all $r\ge 1$ and  $\theta\in\Theta_{\rm in}$ 
\[
\liminf_{\alpha_{\max}\to 0}  \inf_{\Dsf\in\class(\alphab)}\Eb_\theta [\tau_\theta(\varepsilon,\alphab)^r ]  \ge 1,
\]
which completes the proof of the lower bound \eqref{LBmomentscomposite2}.
 \end{proof}
 
 \begin{remark} \label{Rem:AS}
 By  Lemma~1 in~\cite{Tartakov_IMDS_2023}, the right-tail condition  \eqref{RTCond} is satisfied whenever 
 \begin{equation}\label{LLRasComposite}
\frac{\lambda_{\theta, \vartheta}(n)}{\psi(n)} \xra[n\to\infty]{\Pb_{\theta}-\text{a.s.}} I(\theta,\vartheta) \quad \text{for all}~~ \theta, \vartheta \in \Theta, \theta \neq \vartheta.
\end{equation}
 Therefore, assertions of Theorem~\ref{Th:LBcomposite} hold under the strong law for the LLR \eqref{LLRasComposite}, which is a natural and more convenient 
 condition.  
Furthermore, as the proof of  Lemma~1 in~\cite{Tartakov_IMDS_2023} shows, for the right-tail condition \eqref{RTCond}  to hold it suffices to assume that
 \[
 \Pb_\theta\brc{\limsup_{n\to \infty}\frac{\lambda_{\theta, \vartheta}(n)}{\psi(n)} > I(\theta,\vartheta)} =0.
 \]
 \end{remark}
 
Consider now the case of simple hypotheses when the parameter $\theta$ takes $N+1$ distinct values $\theta_0, \theta_1, \dots, \theta_N$ or, more generally, 
the distributions $\Pb_i$ under hypotheses $\Hyp_i$, $i=0,1,\dots,N$ have joint distinct densities defined in \eqref{jointdenssimple}. Obviously, in this case, the indifference zone is 
empty $\Theta_{\rm in}=\varnothing$ and the subsets $\Theta_i= \{i\}$, $i=0,1,\dots,N$. Also, the LR and LLR defined in \eqref{LRcomposite} become the LR and LLR processes between
the hypotheses $\Hyp_i$ and~$\Hyp_j$ for the sample $\Xb^n_1$
\begin{equation}\label{LRsimple}
\begin{aligned}
\Lambda_{ij}(n) &= \frac{p_{i,n}(\Xb^n_1)}{p_{j,n}(\Xb^n_1)} = \prod_{t=1}^n \frac{f_{i,t}(X_t | \Xb^{t-1}_1)}{f_{j,t}(X_t | \Xb^{t-1}_1)},
\\
\lambda_{ij}(n) &= \log \Lambda_{ij}(n) = \sum_{t=1}^n \log \brcs{\frac{f_{i,t}(X_t | \Xb^{t-1}_1)}{f_{j,t}(X_t | \Xb^{t-1}_1)}}.
\end{aligned}
\end{equation}
Obviously, in this case,  Theorem~\ref{Th:LBcomposite}  implies the following corollary, whose proof can also be established directly using proofs of Lemma~2.1 and Theorem~2.2 in Tartakovsky~\cite{TartakovskySISP98}. Recall that class $\class_{\rm sim}(\alphab)$ is defined in \eqref{Mclassessimple}.

\begin{corollary} \label{Cor:LBsimple}
If there exist an increasing function $\psi(t)$, $\psi(\infty)=\infty$, satisfying condition \eqref{CondPsi},  and positive and finite numbers $I_{ij}$, $i,j=0,1,\dots,N$, $i\neq j$ 
such that for all 
$\varepsilon>0$ and all  $i,j=0,1,\dots,N$ ($i\neq j$)
\begin{equation}\label{Probmaxsimple}
\lim_{L\to \infty} \Pb_i \set{\frac{1}{\psi(L)}\max_{1 \le  n \le L} \lambda_{ij}(n) \ge (1+\varepsilon) I_{ij}} =0 ,
\end{equation}
then for all $i=0,1,\dots,N$ and every $0 < \varepsilon < 1$
\begin{equation}\label{Prob1simple}
\lim_{\alpha_{\max}\to 0} \inf_{\Dsf\in\class(\alphab)}  \Pb_i\set{T >  \Psi\brc{(1-\varepsilon)\max_{j\in \Ni} \frac{|\log \alpha_{ji}|}{I_{ij}}}} =1,
\end{equation}
and therefore,  for all $r \ge 1$ and $i=0,1,\dots,N$
\begin{equation}\label{LBmomentssimple}
\inf_{\Dsf\in \class_{\rm sim}(\alphab)} \Eb_i[T^r] \ge \brcs{\Psi\brc{ \max_{j\in \Ni} \frac{|\log \alpha_{ji}|}{I_{ij}} }}^r (1+o(1)) \quad  \text{as}~~ \alpha_{\max} \to 0.
\end{equation}
\end{corollary}

 \begin{remark} \label{Rem:ASsimple}
 The right-tail condition  \eqref{Probmaxsimple} is satisfied whenever 
  \[
 \Pb_i\brc{\limsup_{n\to \infty}\frac{\lambda_{ij}(n)}{\psi(n)} \le I_{ij}} =1,
 \]
 and therefore, if normalized LLRs $\lambda_{ij}(n)/\psi(n)$ converge $\Pb_i$-a.s.\ to $I_{ij}$ as $n\to\infty$.
 \end{remark}

\section{Asymptotic optimality of likelihood ratio based sequential multi-decision rules for simple hypotheses} \label{s:Rulessimple}

We begin with the case of $N+1$ simple hypotheses $\Hyp_i:\Pb=\Pb_i$, $i=0,1,\dots,N$ with joint densities \eqref{jointdenssimple}. 

\subsection{The matrix sequential probability ratio test}\label{ss:MSPRT}

Recall that $\Nc_0=\{0,1,\dots,N\}$. For a threshold matrix $(A_{ij})_{i,j\in\Nc_0}$, with $A_{ij} > 0$ and the $A_{ii}$ are immaterial ($0$, say), define 
the Matrix SPRT (\mbox{MSPRT}) $\Dsf_* = (T_*, d_*)$, built on $(N+1)N$ one-sided SPRTs between the hypotheses 
$\Hyp_i$ and~$\Hyp_j$, as follows: Stop at the first $n \ge 1$ such that, for some $i\in \Nc_0$, $\Lambda_{ij}(n) \ge A_{ji}$ for all $j \in \Ni$
and accept the unique~$\Hyp_i$ that satisfies these inequalities.  
Obviously, with $a_{ji}= \log A_{ji}$, introducing the Markov accepting times for the hypotheses~$\Hyp_i$, $i\in \Nc_0$ as
\begin{equation} \label{Ti}
\begin{aligned}
T_i & =
\inf\set{ n \ge 1: \lambda_{ij}(n)\ge a_{ji} ~ \text{for all}~ j \in \Nc_0 \setminus i } ,
\\
& =\inf\set{ n \ge 1: \min_{j \in \Nc_0\setminus i}\brcs{\lambda_{ij}(n) -  a_{ji}} \ge 0},   \quad i =0,1,\dots,N,
\end{aligned}
\end{equation}
the \mbox{MSPRT} $\Dsf_*=(T_*, d_*)$ can be written as
\begin{equation} \label{D1}
T_*=\min_{k \in \Nc_0} T_k, \qquad d_*=i \quad \mbox{if} \quad T_*= T_i.
\end{equation}
Thus, in the \mbox{MSPRT}, each component SPRT is extended until, for some $i\in \Nc_0$, all $N$ SPRTs involving $\Hyp_i$ accept~$\Hyp_i$. Note that for $N=1$ the \mbox{MSPRT} coincides with Wald's SPRT. 

Using Wald's likelihood ratio identity, it can be easily shown that the error probabilities of the \mbox{MSPRT} $\alpha_{ij}(\Dsf_*)=\Pb_i(d_*=j)$ satisfy the inequalities
\begin{equation}\label{PEupperbounds}
\alpha_{ij}(\Dsf_*) \le \exp(-a_{ij}) \quad \text{for all}~~ i,j = 0,1,\dots,N, ~ i \neq  j, 
\end{equation}
so  selecting $a_{ij}  = |\log \alpha_{ij}|$ implies $\Dsf_* \in \class(\alphab)$. See, e.g., Lemma 4.1.1 (page 192) in Tartakovsky {\em et al.} \cite{TNB_book2014}.

In his ingenious paper, Lorden~\cite{Lorden-AS77} showed that 
with a special design that includes accurate estimation of thresholds accounting for overshoots, 
the \mbox{MSPRT} is nearly optimal in the third-order sense -- it minimizes expected sample sizes for all hypotheses up to an additive disappearing term, i.e., 
$\inf_{\Dsf\in\class(\alphab)}\Eb_i[T] = \Eb_i[T_*]+ o(1)$ as $\alpha_{\max}\to0$. This result holds only for i.i.d.\ models with the finite second moment  
$\Eb_i[\lambda_{ij}(1)^2] < \infty$. In the non-i.i.d.\ case, it is practically impossible to obtain such a result, 
so we will focus on the first-order optimality \eqref{FOASoptimsimple}.

To establish the asymptotic optimality property of the \mbox{MSPRT} we use the ideas of the groundbreaking paper of Lai~\cite{Lai-as81-SPRT} where he 
proved first-order asymptotic optimality of Wald's SPRT in the general non-i.i.d.\ case.

By the SLLN in the i.i.d.\ case, the LLR ~$\lambda_{ij}(n)$ has the following stability property
\begin{equation}\label{SLLN_LLR}
n^{-1} \lambda_{ij}(n) \xra[n\to\infty]{\Pb_i-\text{a.s.}} I_{ij}, \quad  i,j = 0,1,\dots,N, ~ i \neq  j,
\end{equation}
where 
\[
I_{ij}= \Eb_i[\lambda_{ij}(1)] = \int \log \brcs{\frac{f_i(x)}{f_j(x)}} f_i(x) \D \mu(x)
\] 
is the Kullback-Leibler (K-L) information number, which characterizes the distance between the hypotheses $\Hyp_i$ and $\Hyp_j$. 
This allows one to conjecture that if in the general non-i.i.d.\ case the LLR is also stable in the sense that the almost 
sure convergence conditions~\eqref{SLLN_LLR} are satisfied with some positive and finite numbers $I_{ij}$, then the \mbox{MSPRT} is approximately optimal. 
 In the general case, these numbers represent the local K-L information in the sense that often (while not always) 
$I_{ij}= \lim_{n\to\infty} n^{-1} \Eb_i[\lambda_{ij}(n)]$.  

Having said that, in what follows, we will assume that the normalized LLRs $\lambda_{ij}(n)/\psi(n)$ converge almost surely to finite and positive numbers $I_{ij}$ under $\Pb_i$:
\begin{equation}\label{LLRasSimple}
\frac{\lambda_{ij}(n)}{\psi(n)} \xra[n\to\infty]{\Pb_i-\text{a.s.}} I_{ij}, \quad  i,j = 0,1,\dots,N, ~ i \neq  j ,
\end{equation}
where $\psi(t)$ is an increasing function ($\psi(\infty)=\infty$) as in Section~\ref{s:ALB}. 

A standard approach for proving asymptotic optimality is to show that the lower bounds \eqref{LBmomentssimple} in Corollary~\ref{Cor:LBsimple} are 
attained for the \mbox{MSPRT} under certain conditions.

\subsection{Asymptotic optimality of the \mbox{MSPRT}} \label{ss:AOMSPRT}

While the almost sure convergence \eqref{LLRasSimple} for the LLRs guarantees lower bounds \eqref{LBmomentssimple}, in the general non-i.i.d.\ case, 
this condition is not sufficient for asymptotic optimality of the \mbox{MSPRT} since it
does not even guarantee the finiteness of the moments~$\Eb_i [T_*^r]$ of the \mbox{MSPRT}'s stopping time. To establish the optimality of the \mbox{MSPRT} a 
strengthening is needed, such as a certain convergence rate in the strong law.

Lai~\cite{Lai-as81-SPRT} proved the asymptotic optimality of Wald's SPRT for testing two hypotheses under the $r$-quick version of the SLLN for the LLR $\lambda(n)/n$, 
i.e., for the models with dependent and asymptotically stationary observations. Tartakovsky~\cite{Tartakovsky-SQA98} generalized Lai's result to the case of multiple 
hypotheses with asymptotically non-stationary observations proving that the \mbox{MSPRT} is asymptotically optimal as long as LLRs $\lambda_{ij}(n)/\psi(n)$ converge 
$r$-quickly to finite numbers $I_{ij}$. Below we relax the $r$-quick convergence condition used in \cite{Lai-as81-SPRT,Tartakovsky-SQA98} by the  $r$-complete convergence.

\begin{definition}
[\noindent\textbf{\textsl{$r$-Complete Convergence}}] We say that the sequence of random variables~$\{Y_n\}_{n\ge 1}$ converges to a random variable~$Y$ 
{\em $r$-completely}  as $n\to\infty$ under probability measure $\Pb$ and write $Y_n \xra[n\to \infty]{\text{$\Pb$-$r$-completely}} Y$ if
\[
\lim_{n\to\infty} \sum_{t=n}^\infty t^{r-1} \Pb(|Y_t-Y| > \varepsilon) =0 \quad \text{for every} ~ \varepsilon >0, 
\]
which is equivalent to 
\[
\sum_{n=1}^\infty n^{r-1} \Pb(|Y_n-Y| > \varepsilon) < \infty \quad \text{for every} ~ \varepsilon >0 .
\]
\end{definition}

As we will see below, a sufficient condition for asymptotic optimality of the \mbox{MSPRT} is the $r$-complete convergence of $\lambda_{ij}(n)/\psi(n)$ to numbers $I_{ij}$, $0< I_{ij}<\infty$ as 
$n\to\infty$ under $\Pb_i$, i.e., that for all $i,j=0,1,\dots,N$ ($i\neq j$)
\begin{equation}\label{rcompleteLLR}
\lim_{n\to\infty} \sum_{t=n}^\infty t^{r-1} \Pb_i\brc{\left |\frac{\lambda_{ij}(t)}{\psi(t)} - I_{ij} \right | > \varepsilon} =0 \quad \text{for every} ~ \varepsilon >0.
\end{equation}

The following theorem provides the asymptotic upper bounds for moments of the sample sizes of the \mbox{MSPRT} which along with the lower bounds 
\eqref{LBmomentssimple}  allow us to conclude that the \mbox{MSPRT} minimizes moments of the sample sizes up to order $r$ in class $\class_{\rm sim}(\alphab)$.

\begin{thm} \label{Th:UBsimple}
Assume that there exist an increasing function $\psi(t)$, $\psi(\infty) = \infty$, satisfying condition \eqref{CondPsi}, and positive and finite numbers $I_{ij}$, $i,j=0,1,\dots,N$, $i\neq j$ 
such that for all  $i,j=0,1,\dots,N$ ($i\neq j$) and some $r \ge 1$
\begin{equation}\label{rcompleteuppersimple}
\lim_{n\to\infty} \sum_{t=n}^\infty t^{r-1} \Pb_i\brc{\frac{\lambda_{ij}(t)}{\psi(t)}  - I_{ij} < -\varepsilon} =0 \quad \text{for every} ~ \varepsilon >0.
\end{equation}
 If the thresholds in the \mbox{MSPRT} are so selected that $\alpha_{ij}(\Dsf_*) \le \alpha_{ij}$ and $a_{ji} \sim \log \alpha_{ji}^{-1}$ as $\alpha_{\max} \to 0$, 
 in particular as $a_{ji} = \log \alpha_{ji}^{-1}$, then for all  $i=0,1,\dots,N$
\begin{equation}\label{UBmomentssimple}
\Eb_i[T_*^r] \le \brcs{\Psi\brc{ \max_{j\in \Ni} \frac{|\log \alpha_{ji}|}{I_{ij}} }}^r (1+o(1)) \quad  \text{as}~~ \alpha_{\max} \to 0.
\end{equation}
\end{thm}

\begin{proof}
For $i,j=0,1,\dots,N$ and $0<\varepsilon < \max_{i,j\in \Nc_0} I_{ij}$, define $\ab=(a_{ij})_{i,j\in \Nc_0}$ and
\[
M_i(\ab,\varepsilon) = 1 + \Psi\brc{\max_{j\in\Nc_0\setminus i} \frac{a_{ji}}{I_{ij}-\varepsilon}} .
\]
Noting that $\Eb_i[T_*^r] \le \Eb_i[T_i^r]$ and setting $\tau=T_i$ and $N=M_i(\ab,\varepsilon)$ in Lemma~A.1 (Appendix A, page 239) in Tartakovsky~\cite{Tartakovsky_book2020} 
we obtain the inequalities
\begin{equation}\label{ExpTisimple}
\Eb_i[T_*^r] \le \Eb_i[T_i^r] \le M_i^r + r 2^{r-1} \sum_{n=M_i}^\infty n^{r-1} \Pb_i \brc{T_i > n}, ~~ i =0,1,\dots,N.
\end{equation}

By the definition of the Markov time $T_i$ (see \eqref{Ti}), we have 
\begin{align*}
\Pb_{i}(T_i > n)  & = \Pb_{i}\set{\max_{1 \le t \le n} \min_{j \in \Nc_0\setminus i} \brcs{\lambda_{ij}(t) -a_{ji}} <0}
 \\
 & \le \sum_{j \in \Nc_0\setminus i} \Pb_{i}\set{\lambda_{ij}(n) < a_{ji}} ,
\end{align*}
where for $n \ge M_i(\ab,\varepsilon)$ the probability $\Pb_{i}\set{\lambda_{ij}(n) < a_{ji}}$ does not exceed the probability 
$\Pb_i\set{\lambda_{i j}(n)/\psi(n) < I_{i j} - \varepsilon}$, and therefore, for all sufficiently large $n$
\begin{equation}\label{ineqprobbiggernsimple}
\Pb_i(T_i > n) \le \sum_{j \in \Nc_0\setminus i} \Pb_i\set{\frac{\lambda_{ij}(n)}{\psi(n)} < I_{i j} - \varepsilon} .
\end{equation}
Substituting \eqref{ineqprobbiggernsimple} into \eqref{ExpTisimple} yields
\begin{equation*}
\Eb_i[T_*^r] \le  M_i(\ab,\varepsilon)^r + r 2^{r-1} \sum_{j \in \Nc_0\setminus i}\sum_{n=M_i(\ab,\varepsilon)}^\infty n^{r-1} \Pb_i\set{\frac{\lambda_{i j}(n)}{\psi(n)} < I_{i j} 
- \varepsilon}.
\end{equation*}
Since $M_i(\ab,\varepsilon)\to \infty$ as $a_{\min} \to \infty$, by condition \eqref{rcompleteuppersimple}, the second term goes to $0$ and, hence, for all $i\in \Nc_0$
\[
\Eb_i[T_*^r] \le  M_i(\ab,\varepsilon)^r +o(1) ~~ \text{as}~ a_{\min}\to \infty,
\]
where  $\varepsilon$ can be arbitrarily small, so taking the limit $\varepsilon \to 0$ and noticing that due to condition \eqref{CondPsi}
\[
\lim_{\varepsilon\to0} \lim_{a_{\min}\to0} [M_i(\ab,\varepsilon)/M_i(\ab,0)]=1,
\] 
we obtain the asymptotic inequality
\begin{equation}\label{ExpTmsprtAineqasimple}
\Eb_i[T_*^r] \le  \brcs{\Psi\brc{\max_{j\in\Nc_0\setminus i} \frac{a_{ji}}{I_{ij}}}}^r (1+o(1)) \quad \text{as}~ a_{\min}\to \infty.
\end{equation}
Setting $a_{ji} = |\log \alpha_{ji}|$ or, more generally, $a_{ji} \sim |\log \alpha_{ji}|$ (assuming $\alpha_{ij}(\Dsf_*) \le \alpha_{ij}$), 
gives the required inequality \eqref{UBmomentssimple}.
\end{proof}

Corollary \ref{Cor:LBsimple} and Theorem~\ref{Th:UBsimple} give the following first-order asymptotic optimality result.

\begin{thm} \label{Th:AOMSPRTsimple}
Assume that for some  $0<I_{ij}< \infty $, $i,j=0,1,\dots,N$, $i\neq j$ and  increasing function $\psi(n)$, $\psi(\infty) = \infty$, satisfying condition \eqref{CondPsi}, 
the normalized LLRs $\lambda_{ij}(n)/\psi(n)$ converge
$r$-completely to $I_{ij}$ under $\Pb_i$ as $n \to \infty$, i.e., for all  $i,j=0,1,\dots,N$ ($i\neq j$) and some $r \ge 1$ condition \eqref{rcompleteLLR} holds.
 If the thresholds in the \mbox{MSPRT} are so selected that $\alpha_{ij}(\Dsf_*) \le \alpha_{ij}$ and $a_{ji} \sim \log \alpha_{ji}^{-1}$ as $\alpha_{\max} \to 0$, in particular as 
 $a_{ji} = \log \alpha_{ji}^{-1}$, then for all  $i=0,1,\dots,N$
\begin{equation}\label{AOmomentssimple}
\Eb_i[T_*^r] \sim  \brcs{\Psi\brc{\max_{j\in \Ni} \frac{|\log \alpha_{ji}|}{I_{ij}} }}^r \sim \inf_{\Dsf\in\class_{\rm sim}(\alphab)} \Eb_i[T^r] \quad \text{as}~ \alpha_{\max} \to 0.
\end{equation}
\end{thm}

\begin{proof}
Obviously, the $r$-complete convergence condition implies both conditions \eqref{Probmaxsimple} in Corollary~\ref{Cor:LBsimple} and \eqref{rcompleteuppersimple} in 
Theorem~\ref{Th:UBsimple}. Hence, using \eqref{LBmomentssimple} in Corollary~\ref{Cor:LBsimple}, we obtain
\[
\Eb_i[T^r_*] \ge \brcs{\Psi\brc{\max_{j\in \Ni}  \frac{|\log \alpha_{ji}|}{I_{ij}} }}^r (1+o(1)) \quad  \text{as}~~ \alpha_{\max} \to 0.
\]
This inequality along with the reverse inequality \eqref{UBmomentssimple} gives the asymptotic approximation
\[
\Eb_i[T_*^r] = \brcs{\Psi\brc{\max_{j\in \Ni}  \frac{|\log \alpha_{ji}|}{I_{ij}} }}^r (1+o(1)) \quad  \text{as}~~ \alpha_{\max} \to 0,
\]
which, by the lower bound \eqref{LBmomentssimple}, is the best one can do in class $\class_{\rm sim}(\alphab)$.  Thus, both asymptotic equalities in \eqref{AOmomentssimple} follow and the proof is complete.
\end{proof}

\begin{remark}
 Inequalities \eqref{PEupperbounds} for error probabilities imply that the \mbox{MSPRT} belongs to $\class_{\rm sim}(\alphab)$ with $\alpha_{ij} = \exp\{-a_{ij}\}$, 
 so Corollary~\ref{Cor:LBsimple} implies asymptotic lower bounds
\[
\Eb_i[T^r_*] \ge \brcs{\Psi\brc{\max_{j\in \Ni} \frac{a_{ji}}{I_{ij}}}}^r (1+o(1)) \quad  \text{as}~~ a_{\min} \to \infty,
\]
while the upper bounds \eqref{ExpTmsprtAineqasimple}
yield reverse inequalities. Therefore, the following approximations for moments
of the sample sizes of the \mbox{MSPRT} as functions of thresholds $\ab=(a_{ij})$ hold regardless of the probabilities of errors
\[
\Eb_i[T^r_*] = \brcs{\Psi\brc{ \max_{j\in \Ni} \frac{a_{ji}}{I_{ij}} }}^r (1+o(1)) \quad  \text{as}~~ a_{\min} \to \infty.
\]
These approximations can be useful in different problem settings, e.g., in Bayes problems.
\end{remark}

\section{Asymptotic optimality of likelihood ratio based sequential multi-decision rules for composite hypotheses} \label{s:Rulescomposite}

Consider now the problem of testing $N+1$ composite hypotheses $\Hyp_i:\Pb=\Pb_\theta$, $\theta\in \Theta_i$, $i=0,1,\dots,N$ with joint densities \eqref{jointdenscomp}.  
In this case, we focus on the class of tests $\class(\alphab)$ defined in \eqref{Mclassescomp} that confines the maximal error probabilities 
$\sup_{\theta\in\Theta_i} \Pb_\theta(d=j)$. 

Recall that we suppose that the parameter space $\Theta$ is split into $N+2$ disjoint subsets $\Theta_0,\Theta_1,\dots,\Theta_N$, $\Theta_{\rm in}$, so
$\Theta=\sum_{i=0}^N \Theta_i + \Theta_{\rm in}$, where $\Theta_{\rm in}$ is the indifference zone.

\subsection{The matrix mixture sequential probability ratio test}\label{ss:MMSPRT}

Let $\pi$ be a mixing measure (prior distribution for $\theta$) on $\Theta$, $\pi(\theta)>0$ for all $\theta\in\Theta$ and $\int_\Theta \pi(\D\theta)=1$.\footnote{The results
also hold for improper measures with minor constraints if $\Theta$ is a compact set.} For $i\in \Nc_0$  
and $n \ge 1$, define 
\begin{align*}
g_{n}(\Xb^n_1) & = \int_{\Theta}  p_{\theta,n}(\Xb^n_1) \pi(\D \theta) = \int_{\Theta} \prod_{t=1}^n f_{\theta,t}(X_t | \Xb^{t-1}_1) \pi(\D \theta) ;
\\
\hat{g}_{i,n}(\Xb^n_1) &=   \sup_{\theta\in\Theta_j}  p_{\theta,n}(\Xb^n_1) = \sup_{\theta\in\Theta_i}\prod_{t=1}^n f_{\theta,t}(X_t | \Xb^{t-1}_1)
\end{align*}
and introduce the statistics 
\begin{align}
\Lambda^{\pi}_{i}(n) &=  \frac{\int_{\Theta} \prod_{t=1}^n f_{\theta,t}(X_t | \Xb^{t-1}_1) \pi(\D \theta)}{\sup_{\theta\in\Theta_i}\prod_{t=1}^n f_{\theta,t}(X_t | \Xb^{t-1}_1)} =
 \frac{g_{n}(\Xb^n_1) }{\hat{g}_{i,n}(\Xb^n_1) } ,
\label{Lambdapi}
\\
\lambda^{\pi}_{i}(n) &= \log \Lambda^{\pi}_{i}(n) = \log g_{n}(\Xb^n_1)  - \log \hat{g}_{i,n}(\Xb^n_1)  .
\label{lambdapi}
\end{align}

For a $(N+1)\times (N+1)$  matrix $(A_{ij})_{i,j\in\Nc_0}$ of thresholds, with $A_{ij} > 0$ and the $A_{ii}$ are immaterial, define 
the Matrix Mixture SPRT (\mbox{MMSPRT}) $\Dsf_*^{\pi} = (T_*^{\pi}, d_*^{\pi})$ as follows: Stop at the first $n \ge 1$ such that, for some $i\in\Nc_0$, 
$\Lambda_{j}^{\pi}(n) \ge A_{ji}$ for all $j \neq i$ and accept the unique~$\Hyp_i$ that satisfies these inequalities.  
Setting $a_{ji}= \log A_{ji}$ and introducing the Markov accepting times for the hypotheses~$\Hyp_i$, $i=0,1,\dots,N$ as
\begin{equation} \label{Ticomposite}
\begin{aligned}
T_i^\pi  & =\inf\set{ n \ge 1: \lambda_{j}^{\pi}(n)\ge a_{ji} ~ \text{for all}~ j \in \Nc_0 \setminus i } 
\\
& =\inf\set{ n \ge 1: \min_{j \in \Nc_0\setminus i}\brcs{\lambda_{j}^{\pi}(n) -  a_{ji}} \ge 0},   \quad i =0,1,\dots,N,
\end{aligned}
\end{equation}
the \mbox{MMSPRT}  can be written as
\begin{equation} \label{MMSPRT}
T_*^{\pi}=\min_{k \in \Nc_0} T_k^{\pi}, \qquad d_*^{\pi}=i \quad \mbox{if} \quad T_*^{\pi}= T_i^{\pi}.
\end{equation}

\subsection{Error Probabilities of the \mbox{MMSPRT}} \label{ss:PEMMSPRT}

The following lemma provides upper bounds for the error probabilities $\alpha_{ij}(\Dsf_*^\pi, \theta)=\Pb_{\theta}(d_*^\pi=j)$, $\theta\in \Theta_i$ of the \mbox{MMSPRT} $\Dsf_*^{\pi}$. 

\begin{lemma}\label{Lem:PEMMSPRT}
The following upper bounds on the error probabilities of the \mbox{MMSPRT}  hold:
\begin{equation}  \label{UpperPEMMSPRT}
\sup_{\theta\in\Theta_i} \alpha_{ij}(\Dsf_*^{\pi}, \theta)  \le \exp\set{- a_{ij}}  \quad \text{for} ~ i,j =0,1,\dots,N, ~ i \neq  j.
\end{equation}
Therefore, if  $a_{ij}  =  \log(1/\alpha_{ij})$ then  $\Dsf_*^\pi \in \class(\alphab)$ . 
\end{lemma}

\begin{proof}
 Observe that $\{d_*^{\pi}=j\}=\{T_*^{\pi}=T_j^\pi\}$ implies $\{T_j^\pi < \infty\}$ and $\Lambda_{i}^\pi(T_j^\pi)   \ge A_{ij} = e^{a_{ij}}$ on $\{T_j^\pi < \infty\}$. 
 So, for all $\theta\in\Theta_i$ ($i\neq j$), we have the following chain of equalities and inequalities
 \begin{align*}
\alpha_{ij}(\Dsf_*^\pi, \theta) &= \Eb_{\theta} \brcs{\Ind{d_*^{\pi} =j}} \le \Eb_{\theta} \brcs{\Ind{T_j^\pi < \infty}}
\\
& = \Eb_{\theta} \brcs{\Ind{T_j^\pi < \infty} \Lambda_{i}^\pi(T_j^\pi)/\Lambda_{i}^\pi(T_j^\pi)}  
\\
&\le \exp\set{- a_{ij}}  \Eb_{\theta} \brcs{\Ind{T_j^\pi < \infty} \Lambda_{i}^\pi(T_j^\pi)} .
\end{align*}
 Since, obviously,  for any $n\ge 1$ and all~$\theta\in\Theta_i$
 $$
 \Lambda_{i}^\pi(n) \le \int_{\Theta} \Lambda_{\vartheta, \theta}(n) \pi(\D\vartheta)
 $$ 
 and, by the Wald likelihood ratio identity,
 \[
 \Eb_{\theta} \brcs{\Ind{T_j^\pi < \infty}\Lambda_{\vartheta, \theta}(T_j^\pi)} = \Pb_{\vartheta} \brc{T_j^\pi < \infty},
 \]
 we obtain that  for all $i\in \Nc_0\setminus j$
 \begin{align*}
\sup_{\theta\in\Theta_i} \alpha_{ij}(\Dsf_*^\pi, \theta) & \le \exp\set{- a_{ij}} \, \sup_{\theta\in\Theta_i} \int_{\Theta} \Eb_{\theta} \brcs{\Ind{T_j^\pi < \infty}
\Lambda_{\vartheta, \theta}(T_j^\pi)} \pi(\D \vartheta)
 \\
 & = \exp\set{- a_{ij}}  \, \int_\Theta \Pb_{\vartheta} \brc{T_j^\pi < \infty} \pi(\D \vartheta) \le  \exp\set{- a_{ij}} ,
\end{align*}
which proves the upper bound \eqref{UpperPEMMSPRT} and gives the Lemma. 
\end{proof}

 \subsection{First-order uniform asymptotic optimality of the \mbox{MMSPRT}}\label{ss:AOMMSPRT}

For general non-i.i.d.\ models, the SLLN \eqref{LLRasComposite} for the LLR does 
not guarantee the optimality of the \mbox{MMSPRT}. As for simple hypotheses, 
a sort of $r$-complete convergence suffices for the asymptotic optimality.  The following $r$-complete convergence-type conditions for left-tail probabilities are sufficient. For the convenience sake, we write $\theta_i$ for $\theta$  when $\theta$ belongs to  $\Theta_i$. 

For $\theta \notin \Theta_j$, define 
\[
\wtlambda_{\theta, j}(n) =  \log \brcs{\frac{\prod_{t=1}^n f_{\theta,t}(X_t | \Xb^{t-1}_1)}{\sup_{\theta\in\Theta_j}\prod_{t=1}^n f_{\theta,t}(X_t | \Xb^{t-1}_1)}} = 
\log \brcs{\frac{p_{\theta,n}(\Xb^n_1)}{\sup_{\theta\in\Theta_j} p_{\theta,n}(\Xb^n_1)}},
\]
and for $\theta\in\Theta$,
\[
\Upsilon_{\theta, j, \delta, \varepsilon, r}(n)=\sum_{t=n}^\infty t^{r-1} \Pb_{\theta} \set{\inf_{\vartheta\in \Gamma_{\delta,\theta}}\widetilde{\lambda}_{\vartheta, j}(t) 
<(I_{j}(\theta)-\varepsilon) \psi(t)} ,
\]
where $I_j(\theta)=\inf_{\vartheta\in \Theta_j} I(\theta,\vartheta)$ is the ``distance'' between the point $\theta$ and the subset $\Theta_j$, which is assumed strictly positive for all 
$\theta\notin \Theta_j$ (see \eqref{Ipositive}), and
$\Gamma_{\delta,\theta} = \{\vartheta \in \Theta: |\vartheta-\theta| < \delta$\} is the $\delta$-neighborhood of the point $\theta$.

\vspace{2mm}

\noindent $\Cb 2$. \textit{Left-tail Condition}. There exist a positive continuous function $I(\theta,\vartheta)$, satisfying condition \eqref{Ipositive},
such that  for any $\varepsilon >0$ and some $r \ge 1$
\begin{equation}\label{LTCond}
\lim_{n\to \infty} \limsup_{\delta\to 0}  \max_{\stackrel{i,j \in \Nc_0}{ j \neq i}} \sup_{\theta_i\in\Theta_i}\, \Upsilon_{\theta_i, j, \delta, \varepsilon, r}(n) =0 
\end{equation}
and
\begin{equation}\label{LTCond2}
\lim_{n\to \infty} \limsup_{\delta\to 0}  \max_{j \in \Nc_0} \sup_{\theta\in\Theta_{\rm in}}\, \Upsilon_{\theta, j, \delta, \varepsilon, r}(n) =0 .
\end{equation}

Recall that $F_{i,\theta}(\varepsilon, \alphab)$ and $F_{i,\theta}(\alphab)$ are defined in \eqref{Fdefinition} and \eqref{Fdefinition2}, respectively. 

The following theorem establishes the asymptotic upper bounds for moments of the \mbox{MMSPRT} sample sizes which together with the lower bounds \eqref{LBmomentscomposite1} and 
\eqref{LBmomentscomposite2}  imply asymptotic optimality of the \mbox{MMSPRT} in class $\class(\alphab)$.

\begin{thm} \label{Th:UBcomposite}
Assume that there exist an increasing function $\psi(t)$, $\psi(\infty) = \infty$, satisfying condition \eqref{CondPsi}, and a function $I(\theta,\vartheta)$, satisfying \eqref{Ipositive}, 
such that for some $r \ge 1$ left-tail condition $\Cb2$  holds. If the thresholds in the \mbox{MMSPRT} are so selected that 
$\sup_{\theta_i \in \Theta_i}\alpha_{ij}(\Dsf_*^\pi, \theta_i) \le \alpha_{ij}$ and 
$a_{ji} \sim \log \alpha_{ji}^{-1}$ as $\alpha_{\max} \to 0$, in particular as $a_{ji} = \log \alpha_{ji}^{-1}$, then as $\alpha_{\max} \to 0$
\begin{align}
\Eb_\theta[(T_*^\pi)^r] &  \le \brcs{F_{i,\theta}(\alphab)}^r (1+o(1)) ~~ \text{for all}~ \theta\in\Theta_i ~ \text{and}~  i \in \Nc_0;
\label{UBmomentscomposite1}
\\
\Eb_\theta[(T_*^\pi)^r]  & \le \brcs{\min_{0 \le i \le N} F_{i,\theta}(\alphab)}^r (1+o(1))~~ \text{for all}~ \theta\in\Theta_{\rm in} .
\label{UBmomentscomposite2}
\end{align}
\end{thm}

\begin{proof}
By the definition of the Markov time $T_i^\pi$, we have 
\begin{align*}
\Pb_{\theta}(T_i^\pi > n)  & = \Pb_{\theta}\set{\max_{1 \le t \le n} \min_{j \in \Nc_0\setminus i} \brcs{\lambda_{j}^\pi(t) -a_{ji}} <0}
 \\
 & \le \sum_{j \in \Nc_0\setminus i} \Pb_{\theta}\set{\lambda_{j}^\pi(n) < a_{ji}} 
 \\
 & = \sum_{j \in \Nc_0\setminus i} \Pb_{\theta}\set{\log g_{n}(\Xb^n_1) - \log \hat{g}_{j,n}(\Xb^n_1)  < a_{ji}} .
\end{align*}
Since 
\[
\log g_{n}(\Xb^n_1)  \ge \log   \int_{\Gamma_{\delta,\theta}} p_{\vartheta,n}(\Xb^n_1) \pi(\D \vartheta) \ge  \inf_{\vartheta\in \Gamma_{\delta,\theta}} \log p_{\vartheta,n}(\Xb^n_1) + 
\log \pi(\Gamma_{\delta,\theta})
\]
it follows that
\begin{equation}\label{ProbTiupper}
\Pb_{\theta}(T_i^\pi > n)  \le \sum_{j \in \Nc_0\setminus i} \Pb_{\theta}\set{\frac{\inf_{\vartheta\in \Gamma_{\delta,\theta}}\widetilde{\lambda}_{\vartheta, j}(n)}{\psi(n)} < 
\frac{a_{ji} -\log \pi(\Gamma_{\delta,\theta})}{\psi(n)}}.
\end{equation}

Let $\ab=(a_{ij})$ denote the matrix of finite thresholds $a_{ij}$, $i,j=0,1,\dots, N$ ($a_{ii}$ are immaterial) of the \mbox{MMSPRT} and let
\begin{equation}\label{Mitheta}
M_{i,\theta}(\ab,\varepsilon) = 1 + \Psi\brc{\max_{j\in \Ni} \frac{a_{ji}}{I_{j}(\theta)-\varepsilon}}  \quad \text{for}~ \theta\in \Theta.
\end{equation}
It is easily seen that for $n \ge M_{i,\theta}(\ab,\varepsilon)$  the right-hand side in \eqref{ProbTiupper} does not exceed the sum of probabilities  
\[
 \sum_{j \in \Nc_0\setminus i} \Pb_{\theta}\set{\frac{1}{\psi(n)} \inf_{\vartheta\in \Gamma_{\delta,\theta}}\widetilde{\lambda}_{\vartheta, j}(n) <  I_j(\theta) - \varepsilon
+ \frac{1}{\psi(n)} |\log \pi(\Gamma_{\delta,\theta})|}.
\]
Therefore, for all sufficiently large $n$ and $a_{\min}$, for which $|\log \pi(\Gamma_{\delta,\theta})|/\psi(n) < \varepsilon/2$, we obtain
\begin{equation}\label{ineqprobbiggern}
\Pb_{\theta}(T_i^\pi > n) \le \sum_{j \in \Nc_0\setminus i} \Pb_{\theta}\set{\frac{1}{\psi(n)} \inf_{\vartheta\in \Gamma_{\delta,\theta}}\widetilde{\lambda}_{\vartheta, j}(n) <  
I_j(\theta) - \varepsilon/2} .
\end{equation}

Noting that by the definition of the stopping time $T_*^\pi$ in \eqref{MMSPRT} $\Eb_{\theta}[(T_*^\pi)^r] \le \Eb_{\theta}[(T_i^\pi)^r]$ for all $i\in \Nc_0$ and setting 
$\tau=T_i^\pi$ and $N=M_{i,\theta}(\ab,\varepsilon/2)$ in Lemma~A.1 (Appendix A, page 239) in Tartakovsky~\cite{Tartakovsky_book2020}, we obtain the inequalities
\begin{equation}\label{ExpTi}
\Eb_{\theta}[(T_*^\pi)^r] \le \Eb_{\theta}[(T_i^\pi)^r] \le M_{i,\theta}^r + r 2^{r-1} \sum_{n=M_{i,\theta}}^\infty n^{r-1} \Pb_{\theta} \brc{T_i^\pi > n},
\end{equation}
which hold for all $i=0,1,\dots,N$.

Substituting \eqref{ineqprobbiggern} in \eqref{ExpTi} gives
\begin{align} \label{Expthetai}
& \Eb_{\theta}[(T_*^\pi)^r]  \le \Eb_{\theta}[(T_i^\pi)^r] \le   M_{i,\theta}(\ab,\varepsilon/2)^r \nonumber
\\
& \quad + r 2^{r-1} N \max_{j \in \Nc_0\setminus i} \sum_{n=M_{i,\theta}(\ab,\varepsilon/2)}^\infty n^{r-1} \nonumber
  \Pb_{\theta}\set{\inf_{\vartheta\in \Gamma_{\delta,\theta}}\widetilde{\lambda}_{\vartheta, j}(n) < ( I_j(\theta) - \varepsilon/2)\psi(n)} \nonumber
  \\
  & \qquad =  M_{i,\theta}(\ab,\varepsilon/2)^r  + r 2^{r-1} N \max_{j \in \Nc_0\setminus i} \Upsilon_{\theta, j, \delta, \varepsilon, r}(M_{i,\theta}(\ab,\varepsilon/2)) .
\end{align}

Consider the case where $\theta=\theta_i \in \Theta_i$ ($i=0,1,\dots,N$).  Inequality \eqref{Expthetai} implies that for all $i=0,1,\dots,N$
\[
 \Eb_{\theta_i}[(T_*^\pi)^r]  \le M_{i,\theta_i}(\ab,\varepsilon/2)^r  + 
   r 2^{r-1} N  \max_{j \in \Ni} \sup_{\theta_i\in\Theta_i} \Upsilon_{\theta_i, j, \delta, \varepsilon/2, r}(M_{i,\theta_i}(\ab,\varepsilon/2)) .
\]
Since $M_{i,\theta_i}(\ab,\varepsilon/2)\to \infty$ as $a_{\min} \to \infty$, by the left-tail condition \eqref{LTCond}, the second term goes to $0$, so that  for all $i=0,1,\dots,N$
\[
\Eb_{\theta_i}[(T_*^\pi)^r]  \le  M_{i,\theta_i}(\ab,\varepsilon/2)^r +o(1) ~~ \text{as}~ a_{\min} \to \infty.
\]
Since $\varepsilon$ can be arbitrarily  small, taking the limit $\varepsilon \to 0$ and noticing that due to condition \eqref{CondPsi}
\[
\lim_{\varepsilon\to0} \lim_{a_{\min}\to0} [M_{i,\theta_i}(\ab,\varepsilon/2)/M_{i,\theta_i}(\ab,0)]=1,
\] 
we obtain the following asymptotic upper bound
\begin{equation}\label{ExpTmsprtAineqa}
\Eb_{\theta_i}[(T_*^\pi)^r]\le   \brcs{\Psi\brc{\max_{j\in\Nc_0\setminus i} \frac{a_{ji}}{I_{j}(\theta_i)}}}^r (1+o(1)) \quad \text{as}~ a_{\min}\to \infty,
\end{equation}
which holds for all $\theta_i\in\Theta_i$ and all $i=0,1,\dots,N$. 

Setting $a_{ji} = |\log \alpha_{ji}|$ in \eqref{ExpTmsprtAineqa} or, more generally,  $a_{ji} \sim |\log \alpha_{ji}|$ (assuming 
$\sup_{\theta_i\in\Theta_i}\alpha_{ij}(\Dsf_*^\pi, \theta_i) \le \alpha_{ij}$), we obtain the  
inequalities \eqref{UBmomentscomposite1} for all $\theta\in \Theta_i$ and all $i=0,1,\dots,N$.

It remains to prove the inequality \eqref{UBmomentscomposite2} for the indifference zone $\Theta_{\rm in}$.  By inequality \eqref{Expthetai}, for
$\theta\in \Theta_{\rm in}$ and any $i \in \Nc_0$ we have the inequalities
\[
 \Eb_{\theta}[(T_i^\pi)^r] \le M_{i,\theta}(\ab,\varepsilon/2)^r  + r 2^{r-1} N \max_{j \in \Nc_0\setminus i} \Upsilon_{\theta, j, \delta, \varepsilon/2, r}(M_{i,\theta}(\ab,\varepsilon/2)) .
\] 
Now,  note that  $\Eb_\theta[(T_*^\pi)^r] \le \min_{i \in \Nc_0} \Eb_{\theta}[T_i^r]$, and therefore,
\begin{align*}
 \Eb_{\theta}[(T_*^\pi)^r] & \le \min_{i \in \Nc_0} M_{i,\theta}(\ab,\varepsilon/2)^r  
 \\
 \quad &+ 
 r 2^{r-1} N \min_{i\in \Nc_0} \max_{j \in \Nc_0\setminus i} \sup_{\theta \in \Theta_{\rm in}}\Upsilon_{\theta, j, \delta, \varepsilon/2, r}(M_{i,\theta}(\ab,\varepsilon/2)),
\end{align*}
where, by the left-tail condition \eqref{LTCond2}, the second term goes to $0$. Hence, 
\[
 \Eb_{\theta}[(T_*^\pi)^r]  \le  \min_{i\in \Nc_0} M_{i,\theta}(\ab,\varepsilon/2)^r +o(1)~~ \text{as}~ a_{\min} \to \infty.
\]
Since $\varepsilon$ can be arbitrarily  small, taking the limit $\varepsilon \to 0$ and noticing that due to condition \eqref{CondPsi}
\[
\lim_{\varepsilon\to0} \lim_{a_{\min}\to0} \brcs{\min_{i\in \Nc_0} M_{i,\theta}(\ab,\varepsilon/2)/\min_{i\in \Nc_0}M_{i,\theta}(\ab,0)}=1,
\] 
 we obtain the asymptotic upper bound for $\theta \in \Theta_{\rm in}$
\begin{equation}\label{ExpTmsprtAineqa2}
\Eb_{\theta}[T_*^r] \le   \brcs{\Psi\brc{\min_{i \in \Nc_0}\max_{j\in\Nc_0\setminus i} \frac{a_{ji}}{I_{j}(\theta)}}}^r (1+o(1)) \quad \text{as}~ a_{\min}\to \infty.
\end{equation}
Setting $a_{ji} = |\log \alpha_{ji}|$ in \eqref{ExpTmsprtAineqa2} or, more generally,  $a_{ji} \sim |\log \alpha_{ji}|$ (assuming 
$\sup_{\theta_i\in\Theta_i}\alpha_{ij}(\Dsf_*^\pi, \theta_i) \le \alpha_{ij}$), we obtain the  inequality \eqref{UBmomentscomposite2} for $\theta\in \Theta_{\rm in}$.
\end{proof}

Theorems \ref{Th:LBcomposite} and \ref{Th:UBcomposite} give the following first-order asymptotic optimality result.

\begin{thm} \label{Th:AOMMSPRTcomposite}
Assume that there exist an increasing function $\psi(t)$, $\psi(\infty) = \infty$, satisfying condition \eqref{CondPsi},  and a function $I(\theta,\vartheta)$, satisfying \eqref{Ipositive}, 
such that the SLLN for the LLR \eqref{LLRasComposite}
holds, i.e., the normalized LLR  $\lambda_{\theta,\vartheta}(n)/\psi(n)$ converges almost surely to $I(\theta,\vartheta)$ under $\Pb_\theta$ as $n \to \infty$. Assume, 
in addition,  that for some $r \ge 1$ left-tail condition $\Cb 2$ holds. If the thresholds in the \mbox{MMSPRT} are so selected that 
$\sup_{\theta \in \Theta_i}\alpha_{ij}(\Dsf_*^\pi, \theta) \le \alpha_{ij}$ and 
$a_{ji} \sim \log \alpha_{ji}^{-1}$ as $\alpha_{\max} \to 0$, in particular as $a_{ji} = \log \alpha_{ji}^{-1}$, then as $\alpha_{\max} \to 0$
\begin{align}
\inf_{\Dsf \in \class(\alphab)} \Eb_{\theta}[T^r]  & \sim \brcs{F_{i,\theta}(\alphab)}^r \sim \Eb_\theta[(T_*^\pi)^r]  ~~ \text{for all}~ \theta\in\Theta_i ~ \text{and}~  i \in \Nc_0;
\label{AOmomentscomposite1}
\\
\inf_{\Dsf \in \class(\alphab)} \Eb_{\theta}[T^r]  & \sim  \brcs{\min_{0 \le i \le N} F_{i,\theta}(\alphab)}^r \sim \Eb_\theta[(T_*^\pi)^r]  ~~ \text{for all}~ \theta\in\Theta_{\rm in} .
\label{AOmomentscomposite2}
\end{align}
\end{thm}

\begin{proof}
Since the SLLN \eqref{LLRasComposite} implies the right-tail condition \eqref{RTCond} we can use  Theorem~\ref{Th:LBcomposite}  to obtain the asymptotic lower bounds (as $\alpha_{\max}\to0$)
\begin{align*}
\Eb_\theta[(T_*^\pi)^r] & \ge \brcs{\Psi\brc{ \max_{j\in \Ni} \frac{|\log \alpha_{ji}|}{I_{j}(\theta)} }}^r (1+o(1)) ~~ \text{for all}~ \theta\in \Theta_i, i \in \Nc_0;
\\
\Eb_\theta[(T_*^\pi)^r] & \ge \brcs{\Psi\brc{\min_{i \in \Nc_0} \max_{j\in \Ni} \frac{|\log \alpha_{ji}|}{I_{j}(\theta)} }}^r (1+o(1)) ~~ \text{for all}~ \theta\in \Theta_{\rm in}.
\end{align*}
These inequalities along with the reverse inequalities \eqref{UBmomentscomposite1} and \eqref{UBmomentscomposite2} in Theorem~\ref{Th:UBcomposite} yield the asymptotic approximations
\begin{align*}
\Eb_\theta[(T_*^\pi)^r] & = \brcs{\Psi\brc{ \max_{j\in \Ni}  \frac{|\log \alpha_{ji}|}{I_{j}(\theta)} }}^r (1+o(1)) ~~ \text{for all}~ \theta\in \Theta_i, i \in \Nc_0;
\\
\Eb_\theta[(T_*^\pi)^r] & = \brcs{\Psi\brc{ \min_{i \in \Nc_0}\max_{j\in \Ni} \frac{|\log \alpha_{ji}|}{I_{j}(\theta)} }}^r (1+o(1)) ~~ \text{for all}~ \theta\in \Theta_{\rm in},
\end{align*}
so the MMSPRT attains the best lower bounds \eqref{LBmomentscomposite1} and \eqref{LBmomentscomposite2} in class $\class(\alphab)$.  This completes the proof. 
\end{proof}

\begin{remark}\label{Re:MMSPRTthresholds}
 It follows from inequalities \eqref{ExpTmsprtAineqa} and \eqref{ExpTmsprtAineqa2} and the fact that the \mbox{MMSPRT} belongs to 
 $\class(\alphab)$ with $\alpha_{ij} = \exp\{-a_{ij}\}$  that the following approximations (as  $a_{\min} \to \infty$)  for moments of the sample sizes of the 
 \mbox{MMSPRT} as functions of thresholds $\ab=(a_{ij})$ hold regardless of the probabilities of errors
\begin{align*}
\Eb_\theta[(T_*^\pi)^r] & = \brcs{\Psi\brc{ \max_{j\in \Ni}  \frac{a_{ji}}{I_{j}(\theta)} }}^r (1+o(1)) ~~ \text{for all}~ \theta\in \Theta_i, ~ i \in \Nc_0;
\\
\Eb_\theta[(T_*^\pi)^r] & = \brcs{\Psi\brc{ \min_{i \in \Nc_0}\max_{j\in \Ni} \frac{a_{ji}}{I_{j}(\theta)} }}^r (1+o(1)) ~~ \text{for all}~ \theta\in \Theta_{\rm in}.
\end{align*}
These approximations may be useful in various problem settings.
\end{remark}

\begin{remark}
An alternative proof for Theorem~\ref{Th:AOMMSPRTcomposite} is to use $r$-complete convergence of the normalized to $\psi(n)$ decision statistics $\lambda_{ij}^\pi(n)$ 
defined in \eqref{lambdapi} to $I_j(\theta)=\inf_{\vartheta \in \Theta_j} I(\theta,\vartheta)$ as $n\to \infty$ under hypotheses $\Hyp_i$ and 
to $\max_{j \in \Nc_0} I_j(\theta)$ for $\theta\in \Theta_{\rm in}$. That is, replacing condition $\Cb2$ by the condition, 
\[
\lim _{n\to\infty} \sum_{t=n}^\infty n^{r-1} \Pb_\theta \set{\abs{\frac{1}{\psi(n)} \lambda_{ij}^\pi(n) -I_j(\theta)} > \varepsilon} ~~ \text{for}~ \theta\in \Theta_i ~ \text{and all}~
\varepsilon >0,
\]
 where $I_j(\theta)$ is replaced by $\max_{j\in\Nc_0} I_j(\theta)$ for $\theta\in \Theta_{\rm in}$. This latter ``direct'' sufficient condition for 
 asymptotic optimality can be verified in certain interesting examples. 
\end{remark}

\begin{remark}
The proposed \mbox{MMSPRT} can be modified by replacing the statistics $\Lambda_i^\pi$ defined in \eqref{Lambdapi} with the statistics 
\[
\Lambda^{\pi}_{ji}(n) =  \frac{\int_{\Theta_j} \prod_{t=1}^n f_{\theta,t}(X_t | \Xb^{t-1}_1) \pi(\D \theta)}{\sup_{\theta\in\Theta_i}\prod_{t=1}^n f_{\theta,t}(X_t | \Xb^{t-1}_1)} , 
~~ i,j =0,1,\dots,N, ~ i\neq j.
\]
This alternative version of the \mbox{MMSPRT} is also uniformly asymptotically optimal to first order for all hypotheses and parameter values $\theta_i\in\Theta_i$, $i=0,1,\dots,N$.
\end{remark}

\begin{remark}
The above results are also satisfied for improper prior distributions as long as $\Theta$ is a compact set. This is important, for example, in invariant testing problems. 
See Section~\ref{ss:Ex3}. 
\end{remark}

\subsection{Adaptive matrix sequential probability ratio test} \label{ss:AMSPRT}

Let $\hat{\theta}_n=\hat{\theta}_n(\Xb^n_1)$ be an ($\Fc_n$-measurable) estimator of~$\theta$. If in conditional density~$f_{\theta,t}(X_t|\Xb^{t-1}_1)$ for the $t$-th observation 
we replace the parameter with the estimator~$\hat{\theta}_{t-1}(\Xb^{t-1}_1)$ built upon the sample $\Xb^{t-1}_1$ of size $t-1$ that includes $t-1$ observations,
then $f_{\hat\theta_{t-1},t}(X_t|\Xb^{t-1}_1)$ and $\prod_{t=1}^n f_{\hat\theta_{t-1}, t}(X_t|\Xb^{t-1}_1)$
are still viable probability densities in contrast to the popular generalized likelihood ratio approach when maximization over $\theta$
is performed in the density $p_\theta(\Xb^n_1)$ for the whole sample $\Xb^n_1$ containing all $n$ available observations, so that $p_{\hat\theta_n}(\Xb^n_1)$ is
not a probability density anymore.\footnote{In the latter case, the $\hat\theta_n$ is the maximum likelihood estimator.} 

For $n\ge 1$, introduce the adaptive LR statistic
\begin{equation*}
\hLa_\vartheta(n) = \prod_{t=1}^n\frac{f_{\hat\theta_{t-1},t}(X_t |\Xb^{t-1}_1)}{f_{\vartheta, t}(X_t|\Xb^{t-1}_1)} , ~~ \hLa_\vartheta(0)=1.
\end{equation*}
It satisfies the recursion
\begin{equation}\label{ALRrec}
\hLa_\vartheta(n) =\hLa_\vartheta(n-1) \times \frac{f_{\hat\theta_{n-1}, n}(X_n|\Xb^{n-1})}{f_{\vartheta, n}(X_n | \Xb^{n-1})}, ~~ n \ge 1, ~ \hLa_\vartheta(0)=1
\end{equation}
with $f_{\hat\theta_{0},1}(X_1|\Xb^0)=p_{\theta_0}(X_1)$, where the initial value of the estimator $\hat\theta_0=\theta_0$ is a design parameter. 
Since  
\[
\Eb_{\vartheta} [\hLa_\vartheta(n)| \Xb^{n-1}] =\hLa_\vartheta(n-1), ~~ n \ge 2, ~~  \Eb_{\vartheta} [\hLa_\vartheta(1)] =1
\] 
the adaptive LR $(\hLa_\vartheta(n), \Fc_n)_{n\geq 1}$ is a $\Pb_{\vartheta}$-martingale with unit expectation. This important property allows us to deduce simple 
upper bounds for error probabilities of the adaptive multi-hypothesis sequential test, introduced below, using the Wald-Doob likelihood ratio identity, 
as we will see in the next subsection. 

Introduce the statistics 
\begin{equation}\label{ALRmulti1}
\hLa_i^*(n)  =  \frac{\prod_{t=1}^n f_{\hat\theta_{t-1}, t}(X_t|\Xb^{t-1}_1)}{\sup_{\theta\in \Theta_i} \prod_{t=1}^n f_{\theta, t}(X_t|\Xb^{t-1}_1)} , \quad i=0,1,\dots,N ,
\end{equation}
and based on these statistics for a $(N+1)\times (N+1)$  matrix $(A_{ij})_{i,j\in\Nc_0}$ of boundaries, with $A_{ij} > 0$ and the $A_{ii}$ are immaterial, define 
the Adaptive Matrix SPRT (\mbox{AMSPRT}) $\widehat\Dsf = (\hT, \hat{d})$ as follows: Stop at the first $n \ge 1$ such that, for some $i\in\Nc_0$, $\hLa^*_{j}(n) \ge A_{ji}$ for all $j \neq i$
and accept the unique~$\Hyp_i$ that satisfies these inequalities.  Let $\hla^*_i(n) = \log \hLa^*_i(n)$.
Setting $a_{ji}= \log A_{ji}$ and introducing the Markov accepting times for the hypotheses~$\Hyp_i$ as
\begin{equation} \label{hatTicomposite}
\begin{aligned}
\hT_i  & =\inf\set{ n \ge 1: \hla_{i}^*(n)\ge a_{ji} ~ \text{for all}~ j \in \Nc_0 \setminus i } 
\\
& =\inf\set{ n \ge 1: \min_{j \in \Nc_0\setminus i}\brcs{\hla_{i}^*(n) -  a_{ji}} \ge 0},   \quad i =0,1,\dots,N,
\end{aligned}
\end{equation}
the \mbox{AMSPRT} $\widehat\Dsf = (\hT, \hat{d})$  can be written as
\begin{equation} \label{AMSPRT}
\hT=\min_{k \in \Nc_0} \hT_k, \qquad \hat{d}=i \quad \mbox{if} \quad \hT= \hT_i.
\end{equation}

Because of the simple recursive structure of the adaptive LR \eqref{ALRrec}, the \mbox{AMSPRT} is a very attractive alternative to the \mbox{MMSPRT}.  
Robbins and Siegmund \cite{RobbinsSiegmund-Berkeley70,RobbinsSiegmund-AS74} were the first who suggested the idea of using the adaptive LR in the context of 
the so-called power $1$ tests for i.i.d.\ models.

\subsection{Probabilities of errors of the \mbox{AMSPRT}} \label{ss:PEAMSPRT}

The following lemma provides simple upper bounds for the error probabilities $\alpha_{ij}(\widehat\Dsf, \theta)=\Pb_{\theta}(\hat{d}=j)$, $\theta\in \Theta_i$ of the 
\mbox{AMSPRT} \eqref{hatTicomposite}-\eqref{AMSPRT}. 

\begin{lemma}\label{Lem:PEAMSPRT}
The following upper bounds on the error probabilities of the \mbox{AMSPRT}  hold:
\begin{equation*} 
\sup_{\theta\in\Theta_i} \alpha_{ij}(\widehat{\Dsf}, \theta)  \le \exp\set{- a_{ij}}  \quad \text{for} ~ i,j =0,1,\dots,N, ~ i \neq  j.
\end{equation*}
Therefore, if  $a_{ij}  =  \log(1/\alpha_{ij})$ then  $\widehat{\Dsf} \in \class(\alphab)$ . 
\end{lemma}

\begin{proof}
 Observe that the event $\{\hat{d}=j\}=\{\hT=\hT_j\}$ implies the event $\{\hT_j < \infty\}$ and, by the definition of the Markov time $\hT_j$, 
 $\hla_{j}^*(\hT_j)   \ge a_{ij}$ on $\{\hT_j < \infty\}$. So, for all $\theta\in\Theta_i$ ($i\neq j$), we obtain
 \begin{align*}
\alpha_{ij}(\widehat\Dsf, \theta) &= \Eb_{\theta} \brcs{\Ind{\hat{d} =j}} \le \Eb_{\theta} \brcs{\Ind{\hT_j < \infty}}
\\
& = \Eb_{\theta} \brcs{\Ind{\hT_j < \infty} \hLa_{j}^*(\hT_j)\exp\set{- \hla_{j}^*(\hT_j)} } 
\\
&\le \exp\set{-a_{ij}} \, \Eb_{\theta} \brcs{\Ind{\hT_j < \infty} \hLa_{j}^*(\hT_j)} .
\end{align*}
Since $\hLa_{j}^*(n) \le \widehat\Lambda_{\theta}(n)$ for any $n\ge 1$ and all~$\theta\in\Theta_i$, $i\neq j$  
we obtain that for all $\theta\in\Theta_i$ and $i=0,1,\dots, N$
 \begin{align*}
\alpha_{ij}(\widehat\Dsf, \theta)  \le \exp\set{-a_{ij}} \, \Eb_{\theta} \brcs{\Ind{\hT_j < \infty}\widehat\Lambda_{\theta}(\hT_j)} ,
\end{align*}
where, as established in the previous section, the adaptive LR $(\hLa_\theta(n), \Fc_n)_{n\geq 1}$ is the $\Pb_{\theta}$-martingale with expectation 
$\Eb_\theta[\hLa_\theta(n)]=1$. Thus, the adaptive LR $\{\hLa_\theta(n)\}_{n\ge 1}$ is a viable likelihood ratio process, i.e.,   for any $\theta\in \Theta$, 
there exists a probability measure $\widehat{\Pb}_\theta$ such that
\[
\hLa_\theta(n) = \frac{\D \widehat{\Pb}_\theta^{n}}{\D \Pb_\theta^n}, ~~ n \ge 1.
\]
Applying Wald's likelihood ratio identity yields
\[
\Eb_{\theta} \brcs{\Ind{\hT < \infty}\widehat\Lambda_{\theta}(\hT_j)}  = \widehat{\Pb}_\theta(\hT_j < \infty),
\]
and hence, 
\[
\sup_{\theta \in \Theta_i} \alpha_{ij}(\widehat\Dsf, \theta)  \le \exp\set{-a_{ij}} \, \sup_{\theta\in\Theta_i} \widehat{\Pb}_\theta(\hT_j < \infty) \le \exp\set{-a_{ij}}  ,
\]
which gives the lemma. 
\end{proof}

 \subsection{First-order uniform asymptotic optimality of the \mbox{AMSPRT}}\label{ss:AOAMSPRT}

Recall that $I_j(\theta)=\inf_{\vartheta\in \Theta_j} I(\theta,\vartheta)$. For $\theta\in\Theta$, define
\[
\widehat\Upsilon_{\theta, j, \varepsilon, r}(n)=\sum_{t=n}^\infty t^{r-1} \Pb_{\theta} \set{\hla^*_{ j}(t) 
<(I_{j}(\theta)-\varepsilon) \psi(t)} .
\]

To obtain asymptotic upper bounds for moments of the stopping time distribution of the \mbox{AMSPRT} we will use the following left-tail condition.

\vspace{3mm}

\noindent $\Cb 3$. \textit{Adaptive Left-tail Condition}. There exists a positive continuous function $I(\theta,\vartheta)$, satisfying condition \eqref{Ipositive},
such that  for any $\varepsilon >0$ and some $r \ge 1$
\begin{equation}\label{LTCondA}
\lim_{n\to \infty} \max_{\stackrel{i,j \in \Nc_0}{ j \neq i}} \sup_{\theta_i\in\Theta_i}\, \widehat\Upsilon_{\theta_i, j, \varepsilon, r}(n) =0 
\end{equation}
and
\begin{equation}\label{LTCond2A}
\lim_{n\to \infty}  \max_{j \in \Nc_0} \sup_{\theta\in\Theta_{\rm in}}\, \widehat\Upsilon_{\theta, j, \varepsilon, r}(n) =0 .
\end{equation}

The following theorem provides the asymptotic upper bounds for moments of the \mbox{AMSPRT} stopping time distribution which together with the lower 
bounds \eqref{LBmomentscomposite1} and \eqref{LBmomentscomposite2}  imply asymptotic optimality of the \mbox{AMSPRT} in class $\class(\alphab)$.

\begin{thm} \label{Th:UBcompositeAdaptive}
Assume that there exist an increasing function $\psi(t)$, $\psi(\infty) = \infty$, satisfying condition \eqref{CondPsi}, and a function $I(\theta,\vartheta)$, satisfying \eqref{Ipositive}, 
such that for some $r \ge 1$ left-tail 
condition $\Cb3$ holds. If the thresholds in \mbox{AMSPRT} are so selected that $\sup_{\theta \in \Theta_i}\alpha_{ij}(\widehat\Dsf, \theta) \le \alpha_{ij}$ and 
$a_{ji} \sim \log \alpha_{ji}^{-1}$ as $\alpha_{\max} \to 0$, in particular as $a_{ji} = \log \alpha_{ji}^{-1}$, then as $\alpha_{\max} \to 0$
\begin{align}
\Eb_\theta[\hT^r] &  \le \brcs{F_{i,\theta}(\alphab)}^r (1+o(1)) ~~ \text{for all}~ \theta\in\Theta_i ~ \text{and}~  i \in \Nc_0;
\label{UBmomentscomposite1A}
\\
\Eb_\theta[\hT^r]  & \le \brcs{\min_{0 \le i \le N} F_{i,\theta}(\alphab)}^r (1+o(1))~~ \text{for all}~ \theta\in\Theta_{\rm in} .
\label{UBmomentscomposite2A}
\end{align}
\end{thm}

\begin{proof}
By the definition of the Markov time $\hT_i$, we have 
\begin{align*}
\Pb_{\theta}(\hT_i > n)  & = \Pb_{\theta}\set{\max_{1 \le t \le n} \min_{j \in \Nc_0\setminus i}\brcs{\hla_{i}^*(t) -  a_{ji}} < 0} 
 \\
 & \le \Pb_{\theta}\set{\min_{j \in \Nc_0\setminus i}\brcs{\hla_{i}^*(n) -  a_{ji}} < 0}
 \\
 & \le \sum_{j \in \Nc_0\setminus i} \Pb_{\theta}\set{\hla_{i}^*(n)/\psi(n) < a_{ji}/\psi(n)} .
\end{align*}
Let $\ab=(a_{ij})$ denote the matrix of finite thresholds $a_{ij}$, $i,j=0,1,\dots, N$ ($a_{ii}$ are immaterial) of \mbox{AMSPRT} and let $M_{i,\theta}(\ab,\varepsilon) $ be as 
in \eqref{Mitheta}. Obviously, for $n \ge M_{i,\theta}(\ab,\varepsilon)$    
\[
\sum_{j \in \Nc_0\setminus i} \Pb_{\theta}\set{\hla_{j}^*(n)/\psi(n) < a_{ji}/\psi(n)} \le  \sum_{j \in \Nc_0\setminus i} \Pb_{\theta}\set{\frac{1}{\psi(n)} \hla_{j}^*(n)
<  I_j(\theta) - \varepsilon},
\]
and hence,  for all sufficiently large $n$ and $a_{\min}$, we obtain the inequality
\begin{equation}\label{ineqprobbiggernA}
\Pb_{\theta}(\hT_i > n) \le \sum_{j \in \Nc_0\setminus i} \Pb_{\theta}\set{\hla_{j}^*(n)<  (I_j(\theta) - \varepsilon) \psi(n)} .
\end{equation}
Similarly to \eqref{ExpTi}, for all $i=0,1,\dots,N$,
\begin{equation}\label{ExpTiA}
\Eb_{\theta}[\hT^r] \le \Eb_{\theta}[\hT_i^r] \le M_{i,\theta}(\ab,\varepsilon)^r + r 2^{r-1} \sum_{n=M_{i,\theta}(\ab,\varepsilon)}^\infty n^{r-1} \Pb_{\theta} (\hT_i > n).
\end{equation}
Using \eqref{ineqprobbiggernA} and \eqref{ExpTiA} yields
\begin{align} \label{ExpthetaiA}
& \Eb_{\theta}[\hT^r]  \le \Eb_{\theta}[\hT_i^r] \le   M_{i,\theta}(\ab,\varepsilon)^r \nonumber
\\
& \quad + r 2^{r-1} N \sum_{n=M_{i,\theta}(\ab,\varepsilon)}^\infty n^{r-1} 
 \max_{j \in \Nc_0\setminus i}  \Pb_{\theta}\set{\hla_{j}^*(n)< ( I_j(\theta) - \varepsilon)\psi(n)} \nonumber
  \\
  & \qquad =  M_{i,\theta}(\ab,\varepsilon)^r  + r 2^{r-1} N \max_{j \in \Nc_0\setminus i} \widehat\Upsilon_{\theta, j, \varepsilon, r}(M_{i,\theta}(\ab,\varepsilon)) .
\end{align}

Let $\theta=\theta_i \in \Theta_i$ ($i=0,1,\dots,N$).  Then inequality \eqref{ExpthetaiA} implies that for all $i=0,1,\dots,N$
\[
 \Eb_{\theta_i}[\hT^r]  \le M_{i,\theta_i}(\ab,\varepsilon)^r  + 
   r 2^{r-1} N  \max_{j \in \Ni} \sup_{\theta_i\in\Theta_i} \widehat\Upsilon_{\theta_i, j, \varepsilon, r}(M_{i,\theta_i}(\ab,\varepsilon)) .
\]
Since $M_{i,\theta_i}(\ab,\varepsilon)\to \infty$ as $a_{\min} \to \infty$, by the left-tail condition \eqref{LTCondA}, the second term goes to $0$, so that  for all $i=0,1,\dots,N$
\[
\Eb_{\theta_i}[\hT^r]  \le  M_{i,\theta_i}(\ab,\varepsilon)^r +o(1) ~~ \text{as}~ a_{\min} \to \infty.
\]
Since $\varepsilon$ can be arbitrarily  small, taking the limit $\varepsilon \to 0$ and using condition \eqref{CondPsi}
we obtain the following asymptotic upper bound
\begin{equation}\label{ExpTmsprtAineqaA}
\Eb_{\theta_i}[\hT^r]\le   \brcs{\Psi\brc{\max_{j\in\Nc_0\setminus i} \frac{a_{ji}}{I_{j}(\theta_i)}}}^r (1+o(1)) \quad \text{as}~ a_{\min}\to \infty,
\end{equation}
which holds for all $\theta_i\in\Theta_i$ and all $i=0,1,\dots,N$. 

Setting $a_{ji} = |\log \alpha_{ij}|$ in \eqref{ExpTmsprtAineqaA} or, more generally,  $a_{ji} \sim |\log \alpha_{ij}|$ (assuming 
$\sup_{\theta\in\Theta_i}\alpha_{ij}(\widehat\Dsf,\theta) \le \alpha_{ij}$), we obtain the  
inequalities \eqref{UBmomentscomposite1A} for all $\theta\in \Theta_i$ and all $i=0,1,\dots,N$.

Now, let $\theta\in \Theta_{\rm in}$.  Then, by inequality \eqref{ExpthetaiA}, for any $i \in \Nc_0$ 
\[
 \Eb_{\theta}[\hT_i^r] \le M_{i,\theta}(\ab,\varepsilon)^r  + r 2^{r-1} N \max_{j \in \Nc_0\setminus i} \Upsilon_{\theta, j, \varepsilon, r}(M_{i,\theta}(\ab,\varepsilon)) .
\] 
Now,  note that  $\Eb_\theta[\hT^r] \le \min_{i \in \Nc_0} \Eb_{\theta}[\hT_i^r]$, and therefore,
\begin{align*}
 \Eb_{\theta}[\hT^r]  \le \min_{i \in \Nc_0} M_{i,\theta}(\ab,\varepsilon)^r  
 + 
 r 2^{r-1} N \min_{i\in \Nc_0} \max_{j \in \Nc_0\setminus i} \sup_{\theta \in \Theta_{\rm in}}\Upsilon_{\theta, j, \varepsilon, r}(M_{i,\theta}(\ab,\varepsilon)),
\end{align*}
where, by the left-tail condition \eqref{LTCond2A}, the second term goes to $0$. Hence, 
\[
 \Eb_{\theta}[\hT^r]  \le  \min_{i\in \Nc_0} M_{i,\theta}(\ab,\varepsilon)^r +o(1) ~~ \text{as}~ a_{\min} \to \infty.
\]
Since $\varepsilon$ can be arbitrarily  small, taking the limit $\varepsilon \to 0$ and using condition \eqref{CondPsi}
 we obtain the asymptotic upper bound for $\theta \in \Theta_{\rm in}$
\begin{equation*}
\Eb_{\theta}[\hT^r] \le   \brcs{\Psi\brc{\min_{i \in \Nc_0}\max_{j\in\Nc_0\setminus i} \frac{a_{ji}}{I_{j}(\theta)}}}^r (1+o(1)) \quad \text{as}~ a_{\min}\to \infty.
\end{equation*}
Setting $a_{ji} = |\log \alpha_{ij}|$ in this inequality or, more generally,  $a_{ji} \sim |\log \alpha_{ij}|$ (assuming 
$\sup_{\theta\in\Theta_i}\alpha_{ij}(\widehat\Dsf,\theta) \le \alpha_{ij}$) gives
inequality \eqref{UBmomentscomposite2A} for $\theta\in \Theta_{\rm in}$. 
\end{proof}

Theorems \ref{Th:LBcomposite} and \ref{Th:UBcomposite} give the following first-order asymptotic optimality result. 

\begin{thm} \label{Th:AOAMSPRTcomposite}
Assume that there exist an increasing function $\psi(t)$, $\psi(\infty) = \infty$, satisfying condition \eqref{CondPsi},  and a function $I(\theta,\vartheta)$, satisfying \eqref{Ipositive}, 
such that the SLLN for the LLR \eqref{LLRasComposite}
holds, i.e., the normalized LLR  $\lambda_{\theta,\vartheta}(n)/\psi(n)$ converges almost surely to $I(\theta,\vartheta)$ under $\Pb_\theta$ as $n \to \infty$. 
Assume, in addition,  that for some $r \ge 1$ left-tail condition $\Cb 3$ holds. If the thresholds in the \mbox{AMSPRT} are so selected that 
$\sup_{\theta \in \Theta_i}\alpha_{ij}(\widehat\Dsf, \theta) \le \alpha_{ij}$ and 
$a_{ji} \sim \log \alpha_{ji}^{-1}$ as $\alpha_{\max} \to 0$, in particular as $a_{ji} = \log \alpha_{ji}^{-1}$, then as $\alpha_{\max} \to 0$
\begin{align}
\inf_{\Dsf \in \class(\alphab)} \Eb_{\theta}[T^r]  & \sim \brcs{F_{i,\theta}(\alphab)}^r \sim \Eb_\theta[\hT^r]  ~~ \text{for all}~ \theta\in\Theta_i ~ \text{and}~  i \in \Nc_0;
\label{AOmomentscomposite1A}
\\
\inf_{\Dsf \in \class(\alphab)} \Eb_{\theta}[T^r]  & \sim  \brcs{\min_{i \in \Nc_0} F_{i,\theta}(\alphab)}^r \sim \Eb_\theta[\hT^r]  ~~ \text{for all}~ \theta\in\Theta_{\rm in} .
\label{AOmomentscomposite2A}
\end{align}
\end{thm}

The proof is elementary, similar to the proof of Theorem~\ref{Th:AOMMSPRTcomposite}, and is omitted.

\begin{remark}\label{Re:AMSPRTthresholds}
The same reasoning as in Remark~\ref{Re:MMSPRTthresholds} gives the following asymptotic approximations (as  $a_{\min} \to \infty$)  for moments of the sample sizes of the 
 \mbox{AMSPRT} as functions of thresholds $\ab=(a_{ij})$ regardless of the error probabilities
\begin{align*}
\Eb_\theta[\hT^r] & \sim \brcs{\Psi\brc{ \max_{j\in \Ni}  \frac{a_{ji}}{I_{j}(\theta)} }}^r ~~ \text{for all}~ \theta\in \Theta_i, ~ i \in \Nc_0;
\\
\Eb_\theta[\hT^r]  & \sim \brcs{\Psi\brc{ \min_{i \in \Nc_0}\max_{j\in \Ni} \frac{a_{ji}}{I_{j}(\theta)} }}^r ~~ \text{for all}~ \theta\in \Theta_{\rm in}.
\end{align*}
\end{remark}

\section{Examples} \label{s:Examples}

In this section, we consider several examples that are useful for certain practical applications. The substantially non-stationary model for observations in the first example 
turns out to be
adequate for sequential detection of epidemics, as discussed in \cite{LiangTarVeerIEEEIT2023,Tartakovsky&Spivak-Matrix2024} in the context of quickest change-point detection.

\subsection{Example 1: Testing for the mean of normal autoregressive non-stationary process} \label{ss:Ex1}

This example has many applications. In particular, (a) in sensor systems such as radars, acoustic systems, and electro-optic imaging systems where it is required to detect signals with 
unknown intensities from objects in clutter and sensor noise (see, e.g., \cite{Tartakovsky&Brown-IEEEAES08,TNB_book2014}) and (b) in the detection of epidemics, e.g., Covid-19 
(see, e.g., \cite{LiangTarVeerIEEEIT2023,Tartakovsky&Spivak-Matrix2024}).

Observations are of the  form
\begin{equation}\label{MeanARmodel}
X_{n}=\theta \, S_{n}  +\xi_{n},\quad n \ge 1,
\end{equation}
where $S_{n}$ is a deterministic function (e.g., a signal) observed in additive noise $\xi_{n}$ and $\theta \in \Theta=(-\infty, +\infty)$ is an unknown parameter. 
In many applications, noise $\{\xi_{n}\}_{n \ge 1}$ can be adequately modeled by the $p$-th order Gaussian autoregressive process 
AR$(p)$ that satisfies the recursion
\begin{equation}\label{sec:Ex.1}
\xi_{n} = \sum_{t=1}^p \rho_{t} \xi_{n-t} + w_{n}, \quad n \ge 1, 
\end{equation}
where   $\{w_{n}\}_{n\ge 1}$ is an i.i.d.\ normal $\cN(0,\sigma^2)$ sequence ($\sigma>0$). For simplicity, let us set zero initial conditions $\xi_{1-p}=\xi_{2-p}=\cdots=\xi_{0}=0$. 
The coefficients $\rho_{1},\dots,\rho_{p}$ and the variance $\sigma^2$ are known and all roots of the equation $z^p -\rho_{1} z^{p-1} - \cdots - \rho_{p}=0$ 
are in the interior of the unit circle, so that the AR($p$) process is stable.  

For $n \ge1$, define the $p_n$-th order residuals 
\[
\widetilde{S}_{n} = S_{n}- \sum_{t=1}^{p_n} \rho_{t} S_{n-t}, \quad  \widetilde{X}_{n} = X_{n}- \sum_{t=1}^{p_n} \rho_{t} X_{n-t},
\]
where $p_n =p$ if $n > p$ and $p_n =n$ if $1 \le n \le p$.  The conditional density has the form
\begin{align*}
f_{\theta, n}(X_{n} |\Xb^{n-1})&=   \frac{1}{\sqrt{2\pi \sigma^2}} \exp\set{-\frac{(\wtX_{n}-\theta \wtS_{n})^2}{2\sigma^2}},
\end{align*}
and therefore, the LLR 
\begin{equation}\label{LLRAR}
\lambda_{\theta, \vartheta}(n) = \frac{\theta-\vartheta}{\sigma^2}  \sum_{t=1}^{n} \wtS_{t} \wtX_{t} -
\frac{\theta^2-\vartheta^2}{2 \sigma^2} \sum_{t=1}^{n} \wtS_{t}^2 .
\end{equation}
Since under measure $\Pb_{\theta} $ the random variables $\{\wtX_{n}\}_{n\ge 1}$ are independent normal random variables 
$\cN(\theta \wtS_{n},\sigma^2)$, the LLR $\{\lambda_{\theta, \vartheta}(n)\}_{n\ge 1}$ is a $\Pb_\theta$-Gaussian process 
(with independent but non-identically distributed increments) with mean and variance
\begin{equation*}
\Eb_{\theta} [\lambda_{\theta, \vartheta}(n)]   = \frac{1}{2}  \Var_{\theta} [\lambda_{\theta, \vartheta}(n)] =\frac{(\theta-\vartheta)^2}{2\sigma^2}  \sum_{t=1}^{n} \wtS_{t}^2  , ~~
\end{equation*}

Assume that 
\begin{equation}\label{Energy}
\frac{1}{\sigma^2} \lim_{n\to \infty} \frac{1}{\psi(n)}   \sum_{t=1}^{n} \wtS_{t}^2 = Q^2 ,
\end{equation}
where  $0<Q^2<\infty$. In a variety of signal processing applications, this condition holds with $\psi(n)=n$, e.g., in radar applications where the 
signal $S_{n}$ is the sequence of harmonic pulses, in which case $\theta^2 Q^2$ is the so-called signal-to-noise ratio (SNR). In some applications such as detection, 
recognition, and tracking of objects on ballistic
trajectories that can be approximated by polynomials of order $m=2-3$, the function $\psi(n)=n^m$, $m >1$. 
Under condition \eqref{Energy} 
\[
\frac{1}{\psi(n)}\lambda_{\theta, \vartheta}(n) \xra[n\to\infty]{ \Pb_{\theta} -\text{a.s.}} \frac{1}{2} (\theta-\vartheta)^2 Q^2 =I(\theta, \vartheta) ~~ \text{for all}~  \theta\in (-\infty,+\infty),
\]
so that the SLLN \eqref{LLRasComposite} and the right-tail condition $\Cb1$ hold. 

Furthermore, $\lambda_{\theta, \vartheta}(n)/\psi(n)\to I(\theta,\vartheta)$ $r$-completely for all $r\ge 1$.
Indeed, under~$\Pb_\theta$, the ``whitened'' observations can be written as $\wtX_n=\theta\wtS_n + w_n$ and the LLR as
\[
\lambda_{\theta,\vartheta}(n) =(\theta-\vartheta) W_n +  \frac{(\theta-\vartheta)^2}{2\sigma^2} \sum_{t=1}^n \wtS^2_t,
\]
where $W_n=\sigma^{-2} \sum_{t=1}^n \wtS_t w_t$ is a weighted sum of i.i.d.\ normal $\cN(0,\sigma^2)$ random variables $w_t$. Let
\begin{equation}\label{etanormalization}
\eta_{n}=\frac{W_{n}}{\sqrt{\sigma^{-2}\sum_{t=1}^{n} \wtS_{t}^2}} ~~ \text{and}  ~~ b_n(\varepsilon)=
\frac{\varepsilon \psi(n)}{\sqrt{\sigma^{-2}\sum_{t=1}^{n} \wtS_{t}^2}}.
\end{equation}
Note that $\eta_n$ is a standard normal $\cN(0,1)$ random variable and that
$$
\Pb\set{\abs{W_n} >\varepsilon \psi(n)}= \Pb\set{\abs{\eta_n} >b_n(\varepsilon)},
$$
and consequently, for arbitrary $b_n(\varepsilon) >1$
\[
\Pb\set{\abs{W_n} >\varepsilon \psi(n)}=\sqrt{\frac{2}{\pi}}\int^{\infty}_{b_n(\varepsilon)}\,\exp\set{-\frac{u^{2}}{2}}\, \D u \le \exp\set{-\frac{b_n(\varepsilon)^{2}}{2}}.
\]

It follows from assumption \eqref{Energy} 
 that $b_n^2(\varepsilon) \sim \varepsilon^2 \psi(n)/Q^2$ as $n\to\infty$, and therefore, for sufficiently large $n$
$$
\Pb\set{\abs{W_n} >\varepsilon \psi(n)} \le O\brc{\exp\set{-\varepsilon^2 \psi(n)/2 Q^2}}.
$$
Recall that, by assumption \eqref{CondPsi1},  $\lim_{n\to\infty} [\psi(n)/\log n] = \infty$, which along with the previous inequality implies that 
$$
\lim_{n\to\infty}n^{m}\Pb\set{\abs{W_n} >\varepsilon \psi(n)}=0 ~~ \text{for all}~ m>0.
$$
Hence,
\begin{equation} \label{finitesum}
\sum_{n=1}^{\infty} n^{r-1} \Pb\set{\abs{W_n} >\varepsilon \psi(n)} < \infty \quad \text{for all} ~ \varepsilon > 0 ~ \text{and all} ~ r\ge 1 .
\end{equation}
In other words, $W_n/\psi(n)$ converges to~$0$ $r$-completely, which implies that for all $\varepsilon >0$ and all $r \ge 1$
\begin{equation} \label{LLRrcomp}
\sum_{n=1}^{\infty} n^{r-1} \Pb_\theta\set{\abs{\lambda_{\theta,\vartheta}(n) - I(\theta,\vartheta)} >\varepsilon \psi(n)} < \infty ,
\end{equation}
where $I(\theta,\vartheta)= (\theta-\vartheta)^2 Q^2/2$.

If we are interested in testing simple hypotheses $\Hyp_i: \theta=\theta_i$, $i=0,1,\dots,N$, then inequality \eqref{LLRrcomp} implies $r$-complete convergence condition
\eqref{rcompleteLLR} for the LLRs $\lambda_{ij}(n) = \lambda_{\theta_i, \theta_j}(n)$ with $I_{ij} = (\theta_i-\theta_j)^2 Q^2/2$, $i,j\in \Nc_0$, $i \neq j$. Hence, 
by Theorem~\ref{Th:AOMSPRTsimple}, the MSPRT $\Dsf_*=(T_*,d_*)$ is asymptotically optimal, minimizing all positive moments of the sample size to first-order,
and asymptotic approximations \eqref{AOmomentssimple} hold for all $\theta_i$, $i=0,1,\dots,N$ as long as thresholds $a_{ij}$ in the MSPRT are so selected 
that $\alpha_{ij}(\Dsf_*) \le \alpha_{ij}$ and $a_{ji} \sim \log \alpha_{ji}^{-1}$ as $\alpha_{\max} \to 0$.

Next, consider composite hypotheses, and for simplicity, let us focus on two hypotheses $\Hyp_0: \theta \le \theta_0$ and $\Hyp_1: \theta \ge \theta_1$ ($\theta_0 < \theta_1$) 
with the indifference interval $\Theta_{\rm in}=(\theta_0,\theta_1)$ when $N=1$. Generalization for multiple hypotheses is straightforward, but the argument is more cumbersome.
The case of two hypotheses is of special interest in object detection and epidemics detection applications. In the case of two hypotheses, 
the \mbox{MMSPRT} will be referred to as the \mbox{M-2-SPRT}.

To establish the optimality of the \mbox{M-2-SPRT} we need to show that the left-tail $r$-complete convergence condition $\Cb2$ holds. This follows from the argument analogous to 
that used for establishing $r$-complete convergence \eqref{LLRrcomp}. The details are omitted.
 Thus, by Theorem~\ref{Th:AOMMSPRTcomposite}, the \mbox{M-2-SPRT}  $\Dsf_*^\pi$ minimizes as $\alpha_{\max}\to0$ all positive moments of the sample size and
asymptotic formulas \eqref{AOmomentscomposite1} and  \eqref{AOmomentscomposite2} hold for $i,j=0,1$ with 
\begin{equation}\label{q10AR}
\begin{aligned}
I_1(\theta) & = \inf_{\vartheta \ge \theta_1} I(\theta,\vartheta) =  \frac{(\theta_1-\theta)^2 Q^2}{2} \quad \text{for}~ \theta < \theta_1,
\\
I_0(\theta) & = \inf_{\vartheta \le \theta_0} I(\theta,\vartheta) =  \frac{(\theta-\theta_0)^2 Q^2}{2} \quad \text{for}~ \theta > \theta_0 .
\end{aligned}
\end{equation}
Note that in the case of two hypotheses $\alphab=(\alpha_0, \alpha_1)$, where $\alpha_0= \alpha_{01}$, $\alpha_1=\alpha_{10}$, and asymptotic formulas \eqref{AOmomentscomposite1} and  \eqref{AOmomentscomposite2} yield
\begin{equation}\label{AO2MSLRT}
\begin{aligned}
& \inf_{\Dsf\in \class(\alpha_0,\alpha_1)} \Eb_\theta [T^r] \sim  \Eb_\theta [(T_*^\pi)^r] 
\\
\sim &
\begin{cases}
[\Psi\brc{|\log \alpha_0|/I_0(\theta)}]^r & \text{for}~~\theta \ge \theta_1
\\
[\Psi\brc{|\log \alpha_1|/I_1(\theta)}]^r & \text{for}~~\theta \le \theta_0
\\
[\Psi\brc{\min_{i=0,1} |\log \alpha_i|/I_i(\theta)}]^r &  \text{for}~\theta\in (\theta_0,\theta_1)
\end{cases}.
\end{aligned}
\end{equation} 

To prove the asymptotic optimality of the \mbox{AMSPRT} we need to verify condition $\Cb 3$. In the case of two hypotheses, the \mbox{AMSPRT} will be referred to as the 
\mbox{A-2-SPRT}.

Let
\[
\hat\theta_n=\frac{\sum_{t=1}^n \wtS_t \wtX_t}{\sum_{t=1}^n \wtS_t^2} 
\]
be the unconditional MLE of $\theta$ and 
let $\hat\theta_{n,1} = \max(\theta_1,\hat\theta_n)$ and $\hat\theta_{n,0} = \min(\theta_0,\hat\theta_n)$ be MLEs restricted to the sets $\Theta_1=[\theta_1,\infty)$ and 
$\Theta_0=(-\infty,\theta_0]$, respectively. Then the statistics~$\hla_{i}^*(n)$ can be written as
\[
\hla_{i}^*(n) = \frac{1}{\sigma^2} \sum_{t=1}^n (\hat\theta_{t-1} - \hat\theta_{n,i}) \wtS_t \wtX_t  -  \frac{1}{2\sigma^2} \sum_{t=1}^n (\hat\theta_{t-1}^2 - 
\hat\theta_{n,i}^2) \wtS^2_t, \quad i=0, 1 .
\]
In analogy with the argument that has led to \eqref{finitesum}, it can be shown that $r$-completely under~$\Pb_\theta$
\begin{equation}\label{rcompthetahatetc}
\begin{aligned}
\hat\theta_n  & \to \theta, \quad \hat\theta_{n,1} \to \max(\theta_1, \theta), \quad \hat\theta_{n,0} \to \min(\theta_0, \theta),
\\
\hat\theta_n^2  & \to \theta^2, \quad \hat\theta_{n,1}^2 \to \max(\theta_1^2, \theta^2), \quad \hat\theta_{n,0}^2 \to \min(\theta_0^2, \theta^2), 
\\
\frac{1}{\psi(n)} \sum_{t=1}^n \hat\theta_{t-1}^2 \wtS_t^2 & \to \theta^2 \sigma^2 Q^2,  ~~ \frac{1}{\psi(n)} \sum_{t=1}^n \hat\theta_{t-1} \wtS_t \wtX_t \to \theta^2 \sigma^2 Q^2 .
\end{aligned}
\end{equation}
Indeed, we have
\[
\Pb_\theta\set{\abs{\hat\theta_n-\theta} > \varepsilon } = \Pb_\theta\set{\abs{\eta_n} > b_n(\varepsilon)/\psi(n)},
\]
where $\eta_n\sim \cN(0,1)$ and $b_n(\varepsilon)$ are defined in \eqref{etanormalization}.  The same argument that has led to \eqref{finitesum} yields
\[
\sum_{n=1}^\infty n^{r-1} \Pb_\theta\set{\abs{\hat\theta_n-\theta} > \varepsilon \psi(n)} < \infty \quad \text{for all} ~ \varepsilon > 0 ~ \text{and all} ~ r\ge 1,
\]
so 
\begin{equation}\label{rcompthetahat}
\hat\theta_n  \xra[n\to\infty]{\text{$\Pb_\theta$-$r$-completely}} \theta.
\end{equation}
The rest of $r$-complete convergences in \eqref{rcompthetahatetc} are established analogously to \eqref{rcompthetahat}. 
Using \eqref{rcompthetahatetc}, after some manipulations we obtain that for all $r\ge 1$
\begin{equation*}
\begin{aligned}
\frac{1}{\psi(n)} \hla_1^*(n) & \xra[n\to\infty]{\text{$\Pb_\theta$-$r$-completely}} I_1(\theta) ~~ \text{for}~  \theta < \theta_1,
\\
\frac{1}{\psi(n)} \hla_0^*(n) & \xra[n\to\infty]{\text{$\Pb_\theta$-$r$-completely}} I_0(\theta) ~~ \text{for}~  \theta > \theta_0 ,
\end{aligned}
\end{equation*}
where  $I_i(\theta)$'s are given by~\eqref{q10AR}. Hence, condition $\Cb3$ holds.

By Theorem~\ref{Th:AOAMSPRTcomposite}, the \mbox{A-2-SPRT} is asymptotically optimal, minimizing all positive moments of the sample size: 
for all $r \ge1$ as $\alpha_{\max}\to 0$ the asymptotics \eqref{AO2MSLRT} hold with $\hT$. These asymptotic formulas can be also written as
\begin{equation*}
\inf_{\Dsf\in \class(\alphab)} \Eb_\theta [T^r] \sim \Eb_\theta [\hT^r] \sim \begin{cases}
\Psi(2|\log \alpha_1|/[(\theta_1-\theta)^2 Q^2])^r & \text{if}~ \theta \le \theta^*
\\
\Psi(2|\log \alpha_0|/[(\theta-\theta_0)^2 Q^2])^r & \text{if}~ \theta \ge \theta^*,
\end{cases} 
\end{equation*}
where $\theta^* = (\theta_1 \sqrt{c} +\theta_0)/(1+\sqrt{c})$ is the solution of the equation
\[
|\log \alpha_0|/(\theta-\theta_0)^2 = |\log \alpha_1|/(\theta_1-\theta)^2
\]
 and $c\sim(\log\alpha_0)/(\log \alpha_1)$ as $\alpha_{\max}=\max(\alpha_0,\alpha_1) \to 0$ (see  \eqref{Asymcase} in Remark~\ref{Rem:Asymcase}).

In particular, if $S_n=S$, then $\psi(n)=n$, $\Psi(t)=t$, and $Q^2 =(1-\rho_1-\dots - \rho_p)^2 S^2/\sigma^{2}$. In this case, a higher-order approximation to the 
expected sample size $\Eb_\theta[\hT]$ of the A-2-SPRT up to an additive vanishing term $o(1)$ has been obtained in \cite{TarSokBar-IEEETSP2020}.

\subsection{Example 2: Testing for covariance in Gaussian autoregressive models} \label{ss:Ex2}

Consider the problem of testing hypotheses regarding the covariance of the  AR($p$) process $\{X_n\}_{n\ge 1}$ which satisfies the recursion 
\begin{equation}\label{sec:Ex.7}
X_{n} =\sum_{t=1}^{p} \rho_{t}\,X_{n-t}+w_{n} \, , \quad n\ge 1,
\end{equation}
where $\{w_{n}\}_{n\ge 1}$ are i.i.d.\ standard normal $\cN(0,1)$ random variables  and coefficients $\rho_1,\dots,\rho_p$ are unknown. In this case, the parameter $\theta$ is 
$p$-dimensional, $\theta=(\rho_1,\dots,\rho_p)^\top$ where hereafter $\top$ denotes transpose. For $s\ge \ell \ge 1$, write $\Xb_{\ell}^{s}=(X_\ell,\dots,X_s)$.

The conditional density $f_\theta(X_{n}\vert \Xb_1^{n-1})=f_\theta(X_{n}\vert  \Xb_{n-p}^{n-1})$ is 
 \begin{equation*} 
 f_{\theta}(X_{n}\vert \Xb_{n-p}^{n-1})  = \frac{1}{(2\pi)^{p/2}} \,\exp\set{-\dfrac{(\eta_{\theta}(X_n,  \Xb_{n-p}^{n-1}))^{2}}{2}},
\end{equation*}
where  $\eta_{\theta}(y,x)=y-(\theta)^{\top} x$ ($y\in\Rbb$, $x=(x_{1},\ldots,x_{p})\in \Rbb^p$). Thus, for any $\theta\in \Theta = \Rbb^{p}$, the LLR 
$
\lambda_{\theta,\vartheta}(n) = \sum_{t=1}^{n}  \lambda_{\theta,\vartheta}^*(t),
$
where
\begin{align*}
\lambda^*_{\theta,\vartheta}(t)  &=
\log\frac{f_{\theta}(X_t\vert \Xb_{t-p}^{t-1})}{f_{\vartheta}(X_t\vert \Xb_{t-p}^{t-1})}
=
X_t(\theta-\vartheta)^{\top} \Xb_{t-p}^{t-1}+\frac{1}{2} \brcs{(\vartheta^{\top} \Xb_{t-p}^{t-1})^{2}-(\theta^{\top}\Xb_{t-p}^{t-1})^{2}}.
\end{align*}
The process \eqref{sec:Ex.7} is not Markov, but the $p$-dimensional process 
\[
Y_{n}=(X_{n},\ldots,X_{n-p+1})^{\top}\in\Rbb^{p}
\] 
is Markov. 

Next, for any $\theta=(\rho_{1},\ldots,\rho_{p})\in\Rbb^{p}$, define the matrix
\[
L(\theta)= \begin{pmatrix}
\rho_{1} & \rho_2 & \dots & \rho_{p}\\
1 & 0 &  \dots & 0\\
\vdots & \vdots &  \ddots & \vdots \\
0 & 0 &\dots 1&0 
\end{pmatrix}  
\]
and notice that
\begin{equation}\label{Phi}
Y_{n} =  L(\theta) Y_{n-1}+\tilde{w}_{n},  ~~   n \ge 1 ,
\end{equation}
where $\tilde{w}_{n}=(w_{n},0,\ldots,0)^\top \in\Rbb^{p}$.
Obviously,
\[
\Eb[\tilde{w}_{n}\,\tilde{w}^{\top}_{n}] =B=
\begin{pmatrix}
1 & \dots & 0\\
\vdots & \ddots & \vdots\\
0 & \dots & 0
\end{pmatrix} .
\]

Assume that $\theta$ belongs to the set $\Theta_{\rm st}$ for which all eigenvalues $\e_{\ell}(\Lambda)$ of the matrix $L(\theta)$ in modules are less than $1$:
\begin{equation}
\label{set-Theta-i}
\Theta_{\rm st}=\{\theta\in \Rbb^{p}: \max_{1\le \ell \le p}\,\vert\e_{\ell}(L(\theta))\vert<1\}.
\end{equation}
Using \eqref{Phi} it can be shown that in this case the process $\{Y_{n}\}_{n\ge 1}$ is ergodic with stationary 
normal distribution $\cN(0,\F(\theta))$, where 
\[
\F(\theta)=\sum_{n=0}^\infty (\Lambda(\theta))^n B (\Lambda^\top(\theta))^{n}.
\]

Since  $\sup_{t\ge  1}\Eb_\theta\vert X_{t}\vert^{r}<\infty$ for any $r \ge 1$ and $\theta\in \Rbb^p$ it follows that $\sup_{t\ge  1}\Eb_\theta\vert\lambda_{\theta,\vartheta}(t)\vert^{r}<\infty$ 
for any $r \ge 1$ and $\theta\in \Rbb^p$, and therefore,  the SLLN \eqref{LLRasComposite} holds with  $I_{\theta,\vartheta} =  (\theta-\vartheta)^\top \F(\theta)(\theta-\vartheta)/2$ .
Also, using techniques developed in \cite{PergTarSISP2016,Pergametal-JMA2022} it can be shown that the left-tail condition $\Cb2$ is satisfied with
\begin{equation}\label{infIj}
I_j(\theta)=\inf_{\vartheta\in \Theta_j} \frac{1}{2}   (\theta-\vartheta)^\top \F(\theta)(\theta-\vartheta) ~~ \text{for}~ \theta \in \Theta_i ~~ \text{and}~ \theta \in \Theta_{\rm in}.
\end{equation}
 In particular, in the Markov scalar case where $p=1$ and $\theta=\rho_1=\rho$ in \eqref{sec:Ex.7}, we have
\[
I_j(\rho)=\inf_{\rho^*\in \Theta_j}  \frac{(\rho-\rho^*)^{2}}{2(1-\rho^{2})} ~~ \text{for}~ \rho \in \Theta_i ~~ \text{and}~ \rho \in \Theta_{\rm in}. 
\]

By Theorem~\ref{Th:AOMMSPRTcomposite}, the \mbox{MMSPRT}  $\Dsf_*^\pi$ minimizes as $\alpha_{\max}\to0$ all positive moments of the sample size and
asymptotic formulas \eqref{AOmomentscomposite1} and  \eqref{AOmomentscomposite2} hold with $I_j(\theta)$ specified in \eqref{infIj} for any compact subset 
of $\Theta=\Theta_{\rm st}$ defined in \eqref{set-Theta-i}.

\subsection{Example 3: Testing for the mean of Gaussian data with unknown variance} \label{ss:Ex3}

\subsubsection{Multi-hypothesis invariant sequential $t$-test}\label{sss:ttest}

The model discussed in Example 1 has focused on Gaussian data of known variability. A more common practical scenario is when the variability of data is unknown.
In this section, we discuss the simplest i.i.d.\ model, although the results can be extended for more general non-i.i.d.\ situations.

Let $\{X_n\}_{n\ge 1}$ be the sequence of i.i.d.\ normal $\cN(\mu, \sigma^2)$ random variables with unknown mean~$\mu$ and unknown variance~$\sigma^2$, where
the variance~$\sigma^2$ is a nuisance parameter. Let $\theta=\mu/\sigma$. 
We are interested in testing the hypotheses $\Hyp_i: \theta=\theta_i$, $i=0,1,\dots,N$, where $\theta_0,\theta_1,\dots,\theta_N$ are given distinct numbers. 
Lai~\cite{Lai-as81-SPRT}  considered this problem
for testing two hypotheses in the context of invariant tests relative to the unknown variance $\sigma^2$. 
Lai proved that the invariant sequential \mbox{$t$-test} ($t$-SPRT) is 
first-order asymptotically optimal among all tests invariant to $\sigma^2$. Lai's result can be easily extended to multiple hypotheses. Below we show 
that the proposed MSPRT is also asymptotically optimal to first order, minimizing all positive moments of the sample size for all hypotheses in class 
$\class_{\rm sim}(\alphab)$ among all tests invariant under scale changes.
 
 The hypothesis testing problem is invariant under the group of scale changes, i.e., under the transformation which transforms $X_1, X_2,\dots, X_n$ into 
  $cX_1, cX_2,\dots, cX_n$ for an arbitrary non-zero constant $c$. Under this group of transformations,
 the maximal invariant is $\Mi_n=(1, X_2/X_1,\dots, X_n/X_1)$. For $n \ge 1$, let $Y_n = X_n/X_1$,
\[
 \overline{Y}_n= \frac{1}{n}\sum_{t=1}^n Y_t, ~~ v_n^2 =\frac{1}{n} \sum_{t=1}^n Y_t^2, ~~  
 t_n = \frac{\overline{Y}_n}{v_n} = \frac{n^{-1}\sum_{t=1}^n X_t}{[n^{-1}\sum_{t=1}^n X_t^2]^{1/2}}.
\]
Straightforward calculation shows that the density of the maximal invariant under the hypothesis $\Hyp_i$ is
\begin{equation}\label{pMI}
p_i(\Mi_n) = \frac{1}{\sqrt{2\pi{(n-1)} n v_n^{2(n-1)}}} \int_0^{\infty} u^{-1} \exp\left\{n f(u, \theta_i t_n)\right\} \D u,
\end{equation}
where $f(u, z)  = -u^2/2 + z u + \log u$. Therefore, the invariant LLRs are given by
\[
\lambda_{ij}(n) = \log \brcs{\frac{\int_0^{\infty}u^{-1}\exp\set{n f(u, \theta_i t_n)}\D u}{\int_0^{\infty}u^{-1}\exp\set{n f(u, \theta_j t_n)}\D u}}, \quad i,j=0,1,\dots,N, ~ i \neq j.
\]
The invariant MSPRT is defined as in \eqref{Ti}-\eqref{D1} with these invariant LLRs.  Note that the statistic $t_n$ is the famous Student $t$-statistic which is the basis for 
Student's $t$-test in the fixed sample size setting. For this reason, the invariant MSPRT 
 based on $\lambda_{ij}(n, t_n)$ will be referred to as the~$t$-MSPRT. 

Define
\begin{equation*}
 J_n(z) =\int_0^{\infty}u^{-1}\exp\set{n f(u, z)}\,\D u,
\end{equation*}
so the LLRs for the maximal invariant are of the form 
\[
\lambda_{ij}(n)=  \log [J_n(\theta_it_n)/J_n(\theta_j t_n)] , ~~ i,j =0,1,\dots,N, ~ i\neq j.
\]
The invariant LLRs $\lambda_{ij}(n)$ are too complicated for direct use. However, it is possible to replace $\lambda_{ij}(n)$ with a suitable approximation, $\lambda_{ij}(n) \approx\tilde\lambda_{ij}(n)$. 
If $\Pb_i(\lvert \lambda_{ij}(n)-\tilde\lambda_{ij}(n)\rvert<C)=1$ for $n \ge n_0$ ($n_0 \ge 1$) with $C$ a
constant, then the $r$-complete convergence of $n^{-1}\tilde\lambda_{ij}(n)$ to~$I_{ij}$ under~$\Pb_i$ 
implies the $r$-complete convergence $n^{-1}\lambda_{ij}(n) \to I_{ij}$ under~$\Pb_i$.

Specifically, using the uniform version of the Laplace asymptotic integration method (cf.\ Wijsman~\cite{wijsman-AMS71}), it can be shown 
that  uniformly in~$t_n$
\[
 \abs{\lambda_{ij}(n)-n \, g_{ij}(t_n)-\Delta_{ij}(t_n)} \to 0 \quad \text{as}~ n\to\infty,
 \]
where the term~$\Delta_{ij}(t_n)$ is bounded by a finite positive constant $C_{ij}$ and 
\begin{align*}
g_{ij}(t_n)&=\phi(\theta_i t_n)-\phi(\theta_j t_n) - \tfrac{1}{2}~ (\theta_i^2-\theta_j^2),
\\
\phi(t_n) & = \tfrac{1}{4}~ t_n\brc{t_n+\sqrt{4+t_n^2}}+\log\brc{t_n+\sqrt{4+t_n^2}} .
\end{align*}
Consequently,
\begin{equation*}
\abs{n^{-1} \lambda_{ij}(n) -g_{ij}(t_n)} \le C_{ij}/n, \quad n \ge 1,
\end{equation*}

Since $\Eb_i\abs{X_1}^r<\infty$ for all $r \ge 1$, it follows that
\begin{equation*}
t_n \xrightarrow[n\to\infty]{\Pb_{i}-r-\text{completely}} \frac{\Eb_i [X_1]}{\sqrt{\Eb_i [X_1^2]}} = \frac{\theta_i}{\sqrt{1+\theta_i^2}} =Q_i  \quad \text{for all} ~ r \ge 1,
\end{equation*}
and therefore, the normalized LLR $n^{-1} \lambda_{ij}(n)$ converges $r$-completely to~$g_{ij} (Q_i)$ under $\Pb_i$, so that
the $r$-complete convergence conditions~\eqref{rcompleteLLR} hold for all $r\ge 1$ with $\psi(n)=n$ and $I_{ij} = g_{ij}\brc{Q_i}$.

It remains to verify that $I_{ij}>0$. To this end, note that for any fixed $|t| \le 1$, the maximum of the function $\tilde{\phi}(\theta, t)= \phi(\theta t) - \theta^2/2$ over~$\theta$ is attained at 
$\theta^*=t/(1-t^2)^{1/2}$, so that $\theta^*=\theta_i$ if $t=Q_i=\theta_i/(1+\theta_i)^{1/2}$. Hence,
\[
 g_{ij}\brc{Q_i} = \tilde{\phi}(\theta_i, Q_i) - \tilde{\phi}(\theta_j, Q_i)>0.
\]
By Theorem~\ref{Th:AOMSPRTsimple}, the $t$-MSPRT asymptotically minimizes all positive moments of the stopping time and, as $\alpha_{\max}\to0$,
\begin{equation*}
\inf_{\Dsf \in \class_{\rm sim}(\alphab)}\Eb_i[T^r]  \sim \brc{\max_{j\in \Ni} \frac{|\log\alpha_{ji}|}{g_{ij}(Q_i)}}^{r} \sim \Eb_i [T_*^r], ~~ i=0,1,\dots,N.
\end{equation*}

In the case of two hypotheses ($N=1$), the above results are identical to those obtained by Lai \cite{Lai-as81-SPRT} for the $t$-SPRT.

\subsubsection{Adaptive sequential test}\label{sss:atest}

We continue considering the same model as in Subsection~\ref{sss:ttest} but now in the context of the adaptive SPRT. So again 
$X_n\sim \cN(\mu, \sigma^2)$, $n=1,2, \ldots$ are i.i.d.\ normal random variables with unknown mean~$\mu$ and unknown variance~$\sigma^2$, but now we focus on the two
composite hypotheses $\Hyp_0: \mu \le \mu_0, \sigma^2>0$ and $\Hyp_1:\mu \ge \mu_1, \sigma^2>0$, where $\mu_1, \mu_0$ are given numbers, $\mu_1>\mu_0$, 
and $\sigma^2$ is an unknown nuisance parameter. 
If this model is treated in the context of invariant tests when the hypotheses are $\Hyp_i: \mu/\sigma=q_i$, $i=0,1$, where $q_0$ and~$q_1$ are given numbers,
then the results in the previous subsection show that the invariant $t$-SPRT is asymptotically optimal in the class of invariant tests.  However, for values of 
$q=\mu/\sigma$ different from $q_i$, this test is not optimal. It performs especially poorly in the indifference zone~$(q_0,q_1)$. To overcome this drawback 
Tartakovsky~{\it et~al.}~\cite{TNB_book2014} construct an invariant \mbox{$t$-$2$-SPRT}, which minimizes the expected sample size at the worst point $q^*\in (q_0,q_1)$. 
But this test is also not optimal for any other point and performs not great at the points located far from~$q^*$. On the other hand, the \mbox{AMSPRT} 
(which we refer to as the \mbox{A-$2$-SPRT} in the case of two hypotheses) is adaptive and asymptotically efficient at any point~$q \in (-\infty,\infty)$. 
Furthermore, it is also invariant to scale transformations.

Let $\theta=(\mu, \sigma^2)$ and $\tilde\theta=(\tilde\mu, \tilde\sigma^2)$. We now show that all conditions of Theorem~\ref{Th:AOAMSPRTcomposite} (with $N=1$) 
are satisfied when $\{\hat\theta_n\}$ is a sequence of MLEs, which implies uniform asymptotic optimality of  the A-$2$-SPRT with $a_{01}=a_0= \log (1/\alpha_0)$
and $a_{10}=a_1 =\log(1/\alpha_1)$ in class $\class(\alphab)=\class(\alpha_0,\alpha_1)$.

 The LLR  is given by
\begin{align*}
\lambda_{\theta,\tilde{\theta}}(n)&=\frac{n}{2} \log \left(\frac{\tilde{\sigma}^2}{\sigma^2}\right) +\frac{\sigma^2-\tilde{\sigma}^2}{2\tilde{\sigma}^2\sigma^2}\sum_{t=1}^nX_t^2 \\
& \quad +
\frac{\mu\tilde{\sigma}^2-\tilde{\mu}\sigma^2}{\tilde{\sigma}^2\sigma^2}\sum_{t=1}^n X_t -\frac{\mu^2\tilde{\sigma}^2-\tilde{\mu}^2\sigma^2}{2\tilde{\sigma}^2\sigma^2}n,
\end{align*}
and the Kullback--Leibler ``distance'' is
\begin{equation*}
I(\theta,\tilde{\theta})=\Eb_\theta[\lambda_{\theta, \tilde\theta}(1)]=\frac{1}{2}\set{\frac{(\mu-\tilde{\mu})^2+\sigma^2}{\tilde{\sigma}^2} +\log \frac{\tilde{\sigma}^2}{\sigma^2}-1}.
\end{equation*}
By the SLLN, 
\[
n^{-1} \lambda_{\theta,\tilde\theta}(n) \xra[n\to\infty]{\text{$\Pb_\theta$-a.s.}} I(\theta,\tilde{\theta}).
\]
Thus, condition $\Cb1$ holds.  It remains to verify positiveness of 
\[
I_1(\mu, \sigma^2)=\inf_{\substack{\tilde\mu\ge \mu_1,  \tilde\sigma^2>0}}I(\mu, \sigma^2; \tilde{\mu}, \tilde\sigma^2) ~~ \text{for} ~ \mu < \mu_1, \sigma^2>0 
\]
and 
\[
I_0(\mu, \sigma^2)=\inf_{\substack{\tilde\mu\le \mu_0,  \tilde\sigma^2>0}} I(\mu, \sigma^2; \tilde{\mu}, \tilde\sigma^2)   ~~ \text{for} ~ \mu > \mu_0, \sigma^2>0 
\]
and the left-tail condition $\Cb3$.

Let $q=\mu/\sigma$ and $q_i=\mu_i/\sigma$.  Let $\Qbb=(-\infty, + \infty)$ denote the $q$-parameter space and let 
$\Qbb_0=(-\infty, q_0]$, $\Qbb_1=[q_1,\infty)$, $\Qbb_{\rm in}=(q_0,q_1)$.  

The minimum value $\min_{\tilde{\sigma}>0}I(\theta,\tilde{\theta})=\frac{1}{2}\log \left[1+(\mu-\tilde{\mu})^2/\sigma^2\right]$
is achieved at the point $\tilde{\sigma}^2=\sigma^2+(\mu-\tilde{\mu})^2$ and $I_i(\mu, \sigma^2)$  are given by
\begin{equation}\label{Ji}
\begin{aligned}
I_1(\theta) & =\inf_{\substack{\tilde\mu\ge \mu_1, \\ \tilde\sigma^2>0}}I(\theta,\tilde{\theta})=I_1(q) =
\frac{1}{2}\log [1+(q_1-q)^2]  ~~ \text{for} ~ q<q_1 ,
 \\
I_0(\theta)&= \inf_{\substack{\tilde\mu\le \mu_0, \\ \tilde\sigma^2>0}} I(\theta,\tilde{\theta})= I_0(q)=\frac{1}{2}\log[1+ (q-q_0)^2] ~~ \text{for} ~ q > q_0.
\end{aligned}
\end{equation} 
Clearly, $I_0(q) >0$ for $q\in \Qbb_1 + \Qbb_{\rm in} =(q_0,\infty)$ and $I_1(q) >0$ for $q\in \Qbb_0 + \Qbb_{\rm in} = (-\infty, q_1)$, and hence, 
$\min[I_0(q), I_1(q)] > 0$ for $q\in \Qbb_{\rm in} =(q_0,q_1)$.  Therefore, the conditions related to the minimal Kullback--Leibler ``distances'' 
for the corresponding sets hold and it remains to deal with the left-tail $r$-complete convergence condition $\Cb3$. 

The unrestricted MLE 
\[
\hat\theta_n=(\hat{\mu}_n,\hat{\sigma}_n^2)=\arg\sup_{\substack{\mu\in (-\infty,\infty), \\ \sigma^2>0}}\lambda_{\theta; \tilde{\theta}}(n)
\]
is a combination of the sample mean and sample variance,  
\[
\hat\mu_{n} =\overline{X}_n = n^{-1}\sum_{t=1}^n X_t, \quad \hat\sigma_{n}^2=v_n^2=n^{-1}\sum_{t=1}^n (X_t-\overline{X}_n)^2.
\]
 Let 
 \[
 \hat\mu_{n,1}=\max\{\mu_1, \overline{X}_n\}, ~~ \hat\mu_{n,0}=\min\{\mu_0, \overline{X}_n\}, ~~  \hat\sigma_{n, i}^2=n^{-1}\sum_{t=1}^n (X_t-\hat\mu_{n,i})^2
 \]
 be the restricted MLEs of $\mu$ and $\sigma^2$ conditioned on the hypotheses $\Hyp_1$ and~$\Hyp_0$, respectively, 
\begin{align*}
\hat\mu_{n,1}& =\arg\sup_{\substack{\mu\ge \mu_1}}\lambda_{\theta; \tilde{\theta}}(n) , \quad \hat\mu_{n,0}=\arg\sup_{\substack{\mu\le \mu_0}}\lambda_{\theta; \tilde{\theta}}(n),
\\
 \hat\sigma_{n, 1}^2 & = \arg\sup_{\substack{\mu\ge \mu_1, \\ \sigma^2>0}}\lambda_{\theta; \tilde{\theta}}(n), \quad  
 \hat\sigma_{n, 0}^2  = \arg\sup_{\substack{\mu\le \mu_0, \\ \sigma^2>0}}\lambda_{\theta; \tilde{\theta}}(n) .
\end{align*}

Straightforward calculation shows that the decision statistics in the \mbox{A-2-SPRT} are  $\hat\lambda_i^*(n) = \ell(n) - \ell_i(n)$ ($i=0,1$), where
\begin{equation}\label{LLRlambdas}
\begin{aligned}
\ell(n) & = \frac{1}{2} \sum_{t=1}^n \brcs{\log \left(\frac{1}{v_{t-1}^2}\right) + \frac{1}{v_{t-1}^2}\brc{ 2 \overline{X}_{t-1} X_t
-X_t^2  - \overline{X}_{t-1}^2}},
\\
\ell_i(n) &=\frac{n}{2}  \brcs{\log \left(\frac{1}{\hat\sigma_{n, i}^2}\right)  -1}, ~~ i =0,1.
\end{aligned}
\end{equation}
These statistics allow for an efficient recursive computation.  Note that $\ell(n)$ requires an initial condition for the estimate
$\hat\theta_0$. This condition is the design parameter that can be deterministic or random. In particular, 
we could set $\ell(0)=0$.

Since $X_1,X_2,\dots$ are i.i.d.\ and $\Eb_\theta[|X_1|^r]<\infty$ for all $r \ge 1$, it can be shown that the following $r$-complete convergence conditions hold as 
$n\to\infty$ under~$\Pb_\theta$: 
\begin{align*}
\oX_n & \to\mu, \quad \oX^2_n\to\mu^2, \quad v_n^2\to\sigma^2 \quad \text{for all} ~  \mu \in(-\infty, +\infty), ~ \sigma^2>0; \\
\hat\mu_{n,1} & \to \begin{cases}
\mu  & \text{if}~ \mu\ge\mu_1 \\
\mu_1  & \text{if}~ \mu <\mu_1
\end{cases} , \quad
\hat\mu_{n,0}  \to \begin{cases}
\mu & \text{if}~ \mu \le \mu_0 \\
\mu_0 & \text{if}~ \mu >\mu_0
\end{cases} ,
\\
\hat\sigma_{n,1}^2 & \to \begin{cases}
\sigma^2 & \text{if}~ \mu\ge\mu_1 \\
\sigma^2 + (\mu-\mu_1)^2 & \text{if}~ \mu <\mu_1
\end{cases} , \quad
\hat\sigma^2_{n,0}  \to \begin{cases}
\sigma^2 & \text{if}~ \mu \le \mu_0 \\
\sigma^2 + (\mu-\mu_0)^2 & \text{if}~ \mu >\mu_0
\end{cases} .
\end{align*}
Using these relations along with~\eqref{LLRlambdas}, it can be verified that $r$-completely under $\Pb_\theta$ as $n\to\infty$ 
\begin{align*}
 n^{-1}\ell(n) & \to \frac{1}{2} \brc{\log \frac{1}{\sigma^{2}}-1} \quad \text{for all} ~ \mu \in(-\infty, +\infty), ~ \sigma^2>0;\\
n^{-1}\ell_{1}(n) & \to
\begin{cases}
\frac{1}{2}(\log \frac{1}{\sigma^2} -1) & \text{if}~ \mu\ge \mu_1,  ~ \sigma^2>0 \\
\frac{1}{2}(\log \frac{1}{\sigma^2+ (\mu_1-\mu)^2}-1)& \text{if}~ \mu < \mu_1, ~ \sigma^2>0
\end{cases} ,
\\
n^{-1}\ell_{0}(n) & \to
\begin{cases}
\frac{1}{2}(\log \frac{1}{\sigma^{2}} -1) & \text{if}~ \mu\le \mu_0,  ~ \sigma^2>0 \\
\frac{1}{2} (\log \frac{1}{\sigma^2 + (\mu-\mu_0)^2} -1) & \text{if}~ \mu > \mu_0, ~ \sigma^2>0
\end{cases}.
\end{align*} 
Combining these formulas yields (for all $r\ge 1$)
\begin{align*}
n^{-1}\hla_i^*(n)\xra[n\to\infty]{\text{$\Pb_\theta$-$r$-completely}} I_i(\theta)=I_i(q) ~~ \text{for}~ q \in \Qbb \setminus \Qbb_i,~ i=0,1,
\end{align*}
where $I_i(q)$ are given by \eqref{Ji}.

Therefore, condition $\Cb3$ is satisfied with $I_i(\theta)=I_i(q)$. By Theorem~\ref{Th:AOAMSPRTcomposite}, the \mbox{A-2-SPRT} is 
asymptotically optimal, minimizing all positive moments of the sample size in the class of tests $\class(\alpha_0,\alpha_1)$: 
for all $r \ge1$ as $\alpha_{\max}\to 0$ 
\begin{equation*}
\inf_{\Dsf\in \class(\alpha_0,\alpha_1)} \Eb_\theta [T^r] \sim \Eb_\theta [\hT^r] \sim \begin{cases}
(2|\log \alpha_1|/\log [1+(q_1-q)^2])^r & \text{if}~ q \le q^*
\\
(2|\log \alpha_0|/\log [1+(q-q_0)^2])^r& \text{if}~ q \ge q^*,
\end{cases}
\end{equation*}
where $q^*$ is the solution of the equation
\[
\brcs{1+(q_1-q)^2}^c = 1+ (q-q_0)^2
\]
($c\sim(\log\alpha_0)/(\log \alpha_1)$ as $\alpha_{\max} \to 0$). In particular, $q^*=(q_0+q_1)/2$ if $c=1$.

It is also worth noting that the mixture test M-2-SPRT with the mixing improper prior density $\pi(\mu, \sigma) = \sigma^{-1} \D \sigma \D\mu$ is uniformly asymptotically optimal.
To see this, consider for simplicity testing $\mu=q_0=0$ against $\mu \neq 0$ without the indifference zone. Introduce the probability measure
\[
\Pb^\pi= \int_{-\infty}^\infty \int_0^\infty \sigma^{-1} \Pb_{\mu,\sigma} \D \sigma \D\mu.
\]
Then the LR $\Lambda_n^\pi$ of $(X_1,\dots,X_n)$ under $\Pb^\pi$ relative to $\Pb_{0,1}$ is the same as the LR of the maximal invariant $\Mi_n$ under 
$\bar{\Pb}=(2\pi)^{-1/2}\int_{-\infty}^\infty \Pb_q \D q$, i.e.,
\[
\Lambda_n^\pi = \frac{\D \Pb^\pi(\Xb^n_1)}{\D\Pb_{0,1}(\Xb^n_1)} = \frac{1}{\sqrt{2\pi}} \int_{-\infty}^\infty \frac{p_q(\Mi_n)}{p_0(\Mi_n)} \D q,
\]
where $p_q(\Mi_n)$ is as in \eqref{pMI} with $\theta_i=q$. Direct calculation shows that 
\[
\lambda^\pi_n =\log \Lambda^\pi_n = \frac{n}{2} \log \brc{1+ \frac{\overline{X}_n^2}{v_n^2}} - \frac{1}{2} \log n
\]
See \cite{Siegmund_book1985}, page 117.
Clearly,
\[ 
n^{-1}\lambda^\pi_n\xra[n\to\infty]{\text{$\Pb_\theta$-$r$-completely}} I(\theta)=I(q)= \frac{1}{2}\log(1+q^2).
\]
Since $I(q)\equiv I_1(q)\equiv I_0(q)$ when $q_0=q_1=0$ (see \eqref{Ji}), it follows that asymptotic performance of the \mbox{A-2-SPRT} and the \mbox{M-2-SPRT} are the same.

\section{Concluding remarks}

1.  Analogous near-optimality results can be obtained for the multi-hypothesis generalized likelihood ratio SPRT.

2. In the non-i.i.d. cases when observations are severely non-stationary and dependent, computing the LLRs is typically time-consuming since there are no recursive formulas
as in the i.i.d.\ case. Computing the mixtures is even more time-consuming. One way to overcome this difficulty is to use window-limited versions 
when computing the corresponding statistics in a fixed moving time window, following the idea proposed by Lai~\cite{LaiIEEE98} for 
changepoint detection problems.  To ensure the tests exhibit asymptotic optimality properties, the size of the window, denoted as $\tau=\tau(\alpha_{\max})$, must be a 
function of the specified error rate $\alpha_{\max}$. Moreover, it should approach infinity approximately as the maximal value of the optimum expected sample size 
\[
\tau(\alpha_{\max})\sim \max\set{\max_{i\in \Nc_0} \sup_{\theta\in \Theta_i} F_{i,\theta}(\alphab), \min_{i\in \Nc_0} \sup_{\theta\in \Theta_{\rm in}} F_{i,\theta}(\alphab)}.
\]

3. The results of uniform optimality in Theorems~\ref{Th:AOMMSPRTcomposite} and \ref{Th:AOAMSPRTcomposite} can be extended to establish  the 
first-order minimax asymptotic optimality of the \mbox{MMSPRT} $\Dsf_*=(T_*, d_*)$ and the \mbox{AMSPRT} $\widehat\Dsf=(\hT, \hat{d})$. 
That is, both tests solve the asymptotic
version of the Kiefer--Weiss problem \cite{Kiefer&Weiss-AMS1957} of minimizing the expected sample size in the worst scenario with respect to the parameter $\theta$, 
or more generally, minimizing higher moments of the stopping time in the worst case. Specifically, let
\[
 \tilde{I}(\theta)= \begin{cases}
\min\limits_{j\in \Nc_0\setminus i} I_j(\theta) & \text{if}~\theta\in\Theta_i
\\
\max\limits_{0 \le i \le N} \min\limits_{j\in \Nc_0\setminus i} I_j(\theta) & \text{if}~\theta\in\Iin
\end{cases}.
\]
It can be shown that if instead of conditions \eqref{Ipositive} we will require a stronger separability condition $\inf_{\theta\in \Theta} \tilde{I}(\theta) >0$
then as $\alpha_{\max}\to 0$
\[
\inf_{\Dsf\in \class(\alphab)} \sup_{\theta\in\Theta} \Eb_\theta [T^r] \sim   \brcs{\max\set{\max_{i\in\Nc_0}\sup_{\theta \in \Theta_i} F_{i,\theta}(\alphab), 
\sup_{\theta\in \Theta_{\rm in}}\min_{i \in \Nc_0} F_{i,\theta}(\alphab)}}^r ,
\]
and the right-hand side is attained for $\sup_{\theta\in\Theta} \Eb_\theta [T_*^r]$ and $\sup_{\theta\in\Theta} \Eb_\theta [\hT^r]$. In the i.i.d.\ case, this problem
has been addressed by Lai~\cite{Lai-AS73}, Lorden~\cite{lorden-as76},  Huffman~\cite{Huffman-AS83}, Pavlov~\cite{Pavlov2} among others and in the non-i.i.d.\ case
by Tartakovsky~{\em et al.}~\cite{TNB_book2014}.

4.  As stated in previous sections, the almost sure convergence of the normalized LLR $\lambda_{\theta,\vartheta}(n)/\psi(n) \to I(\theta,\vartheta)$ as $n\to\infty$ 
under $\Pb_\theta$ is not sufficient for the optimality of proposed sequential tests in the sense of minimizing the expected sample size or moments of the sample size. 
However, the following  weak (in the almost sure sense)  asymptotic optimality holds under this almost sure convergence condition for the \mbox{MMSPRT} $(T_*,d_*)$ and the \mbox{AMSPRT} $(\hT, \hat{d})$:
for all $\theta\in \Theta_i$ and $i =0,1,\dots,N$ 
\[
\lim_{\alpha_{\max} \to 0} \sup_{T\in \class(\alphab)} \Pb_\theta\set{\tau  \ge (1+\varepsilon) T} =0 ~~ \text{for every} ~  \varepsilon>0
\] 
and
\[
\frac{\tau}{F_{i,\theta}(\alphab)} \xra[\alpha_{\max} \to 0]{\Pb_\theta-\text{a.s.}} 1,
\]
where $\tau=T_*$ or $\tau=\hT$. Lai~\cite{Lai-as81-SPRT} established this result for the SPRT in the problem of testing two simple hypotheses and an asymptotically stationary case 
($\psi(n)=n$).

\section*{Acknowledgements}

This research draws inspiration from Tze Lai's seminal paper~\cite{Lai-as81-SPRT}, which explored the asymptotic optimality of the SPRT for general non-i.i.d. models. 
The findings presented in Section~\ref{s:Rulessimple}, regarding the first-order asymptotic optimality of the multi-hypothesis (matrix) SPRT, directly extend Lai's contributions. 
I am deeply grateful to Tze Lai for the insightful conversations we shared between 1993 and 2023, which significantly contributed to this work. 


%
%


\begin{thebibliography}{999}

  
 \bibitem[Chan and Lai(2000)]{ChanLai-AS2000}
H.~P. Chan and T.~L. Lai,
\newblock Asymptotic approximations for error probabilities of sequential or fixed sample size tests in exponential families,
\newblock \emph{Annals of Statistics} {\bf 28} (2000), 1638--1669.

\bibitem[Fellouris and Tartakovsky(2017)]{FellourisTartakovsky-IEEEIT2017}
G.\ Fellouris  and A. G.\ Tartakovsky, 
\newblock Multichannel sequential detection -- Part I: Non-i.i.d. data,
\newblock {\em IEEE Transactions on Information Theory}  {\bf 63} (2017), 4551--4571.
\newblock {\url{https://doi.org/10.1109/TIT.2017.2689785}}.

\bibitem{Huffman-AS83}
M. D. Huffman,
\newblock An efficient approximate solution to the Kiefer--Weiss problem,
\newblock {\em Annals of Statistics} {\bf 11} (1983), 306--316.

\bibitem{Kiefer&Weiss-AMS1957}
J.~Kiefer and L.~Weiss,
\newblock Properties of generalized sequential probability ratio tests,
\newblock {\em Annals of Mathematical Statistics} {\bf 28} (1957), 57--74.

\bibitem{Lai-AS73}
T. L. Lai,
\newblock Optimal stopping and sequential tests which minimize the maximum expected sample size,
\newblock {\em Annals of Statistics} {\bf 1} (1973), 659--673.

\bibitem[Lai(1981)]{Lai-as81-SPRT}
T. L. Lai, 
\newblock Asymptotic optimality of invariant sequential probability ratio tests,
\newblock {\em Annals of Statistics} {\bf 9} (1981), 318--333.

\bibitem[Lai(1988)]{Lai-AS1988}
T. L. Lai,
\newblock Nearly optimal sequential tests of composite hypotheses,
\newblock \emph{Annals of Statistics} {\bf 16} (1988), 856--886.

\bibitem[Lai and Zhang(1994)]{LaiZhang-SQA1994}
T. L. Lai and L.~Zhang,
\newblock {A modification of Schwarz's sequential likelihood ratio tests in mulitvariate sequential analysis},
\newblock \emph{Sequential Analysis} {\bf 13} (1994), 79--96.

\bibitem[Lai(1998)]{LaiIEEE98}
T. L. Lai,  Information bounds and quick detection of parameter changes in stochastic systems,  
\emph{IEEE Transactions on Information Theory} {\bf 44} (1998), 2917--2929.

\bibitem[Liang \em{et~al.}(2023)Liang, Tartakovsky, and
  Veeravalli]{LiangTarVeerIEEEIT2023}
Y. Liang,  A. G. Tartakovsky and V. V. Veeravalli, 
\newblock Quickest change detection with non-stationary post-change observations,
\newblock {\em IEEE Transactions on Information Theory}  {\bf 69} (2023), 3400--3414.
\newblock {\url{https://doi.org/0.1109/TIT.2022.3230583}}.

\bibitem[Lorden(1967)]{Lorden1}
G. Lorden,
\newblock Integrated risk of asymptotically {Bayes} sequential tests,
\newblock \emph{Annals of Mathematical Statistics} {\bf 38} (1967), 1399--1422.

\bibitem[Lorden(1973)]{lorden-as73}
G. Lorden,
\newblock {Open-ended tests for Koopman--Darmois families},
\newblock \emph{Annals of Statistics} {\bf 1} (1973),  633--643.

\bibitem[Lorden(1976)]{lorden-as76}
G. Lorden,
\newblock {2-SPRT's} and the modified Kiefer-Weiss problem of minimizing an expected sample size,
\newblock \emph{Annals of Statistics} {\bf 4} (1976), 281--291.

\bibitem[Lorden(1977)]{Lorden-AS77}
G. Lorden, 
\newblock Nearly-optimal sequential tests for finitely many parameter values,
\newblock {\em Annals of Statistics} {\bf 5} (1977), 1--21.

\bibitem[Lorden(1977)]{Lorden-unpublished-1977}
G. Lorden,
\newblock Nearly optimal sequential tests for exponential families,
\newblock \emph{Unpublished Manuscript} (1977). Available from \url{http://jaybartroff.com/research/gary.pdf}

\bibitem[Pavlov(1990)]{Pavlov2}
I. V. Pavlov, 
\newblock {Sequential procedure of testing composite hypotheses with applications to the Kiefer-Weiss problem},
\newblock {\em Theory of Probability and its Applications}  {\bf 35} (1990), 280--292.  

\bibitem[Pergamenchtchikov and Tartakovsky(2018)]{PergTarSISP2016}
S. Pergamenchtchikov and A. G.  Tartakovsky, 
\newblock Asymptotically optimal pointwise and minimax quickest change-point detection for dependent data,
\newblock {\em Statistical Inference for Stochastic Processes}  {\bf 21} (2018), 217--259.

\bibitem[Pergamenchtchikov, Tartakovsky, and Spivak(2022)]{Pergametal-JMA2022}
S. M. Pergamenchtchikov, A. G. Tartakovsky and V. S. Spivak,  Minimax and pointwise
sequential changepoint detection and identification for general stochastic models,
\newblock {\em Journal of Multivariate Analysis} \textbf{190} (2022), 1--22.

\bibitem[Robbins and Siegmund(1972)]{RobbinsSiegmund-Berkeley70}
H. Robbins and D. Siegmund,
\newblock A class of stopping rules for testing parameter hypotheses,
\newblock In L.~M. Le Cam, J.~Neyman, and E.~L. Scott, editors,
  \emph{Proceedings of the Sixth Berkeley Symposium on Mathematical Statistics
  and Probability}, volume 4: Biology and Health, pages
  37--41. University of California Press, Berkeley, CA, USA, 1972, 37--41.

\bibitem[Robbins and Siegmund(1974)]{RobbinsSiegmund-AS74}
H. Robbins and D. Siegmund,
\newblock The expected sample size of some tests of power one,
\newblock \emph{Annals of Statistics} {\bf 2} (1974), 415--436.

\bibitem[Siegmund(1985)]{Siegmund_book1985}
D. Siegmund, {\em Sequential Analysis: Tests and Confidence Intervals}, Springer-Verlag, New York, Berlin, 1985.

\bibitem[Siegmund(2013)]{Siegmund2013}
D. Siegmund, 
\newblock Change-points: from sequential detection to biology and back,
\newblock {\em Sequential Analysis} {\bf 32} (2013), 2--14.

\bibitem[Schwarz(1962)]{Schwarz-AMS1962}
G. Schwarz,
\newblock {Asymptotic shapes of Bayes sequential testing regions},
\newblock \emph{Annals of Mathematical Statistics} {\bf 33} (1962), 224--236.

\bibitem[Tartakovsky(1998a)]{Tartakovsky-SQA98}
A. G. Tartakovsky, 
\newblock Asymptotically optimal sequential tests for nonhomogeneous processes,
\newblock {\em Sequential Analysis} {\bf 17} (1998), 33--62.

\bibitem[Tartakovsky(1998b)]{TartakovskySISP98}
A. G. Tartakovsky,
\newblock Asymptotic optimality of certain multihypothesis sequential tests: Non-i.i.d. case,
\newblock {\em Statistical Inference for Stochastic Processes}  {\bf 1} (1998), 265--295.

\bibitem[Tartakovsky(2020)]{Tartakovsky_book2020}
A. G. Tartakovsky,
\newblock {\em Sequential Change Detection and Hypothesis Testing: General
  Non-i.i.d. Stochastic Models and Asymptotically Optimal Rules}, Monographs on
  Statistics and Applied Probability 165, Chapman \& Hall/CRC Press, Taylor \&
  Francis Group, Boca Raton, London, New York,  2020.
  
  \bibitem[Tartakovsky(2023)]{Tartakov_IMDS_2023}
A. G. Tartakovsky, Quick and complete convergence in the law of large numbers with applications to statistics, 
\textit{MDPI Mathematics} \textbf{11} (2023), 2687.

\bibitem[Tartakovsky \em{et~al.}(2015)Tartakovsky, Nikiforov, and
  Basseville]{TNB_book2014}
A. G. Tartakovsky, I. V. Nikiforov and M. Basseville, 
\newblock {\em Sequential Analysis: Hypothesis Testing and Changepoint
  Detection}, Monographs on Statistics and Applied Probability 136, Chapman \&
  Hall/CRC Press, Taylor \& Francis Group, Boca Raton, London, New York,  2015.
  
\bibitem[Tartakovsky, Sokolov and Bar-Shalom]{TarSokBar-IEEETSP2020}   
A. G. Tartakovsky, G. Sokolov and Y. Bar-Shalom, Nearly optimal adaptive sequential tests for object detection,
{\em IEEE Transactions on Signal Processing} {\bf 68} (2020),  pp. 3371--3384. 
DOI: 10.1109/TSP.2020.2986542.  
  
\bibitem[Tartakovsky and Spivak(2024)]{Tartakovsky&Spivak-Matrix2024}  
A. G. Tartakovsky and V. Spivak, Quickest changepoint detection in general multistream stochastic models: recent results, 
applications and future challenges, {\em Matrix Annals}, 2024 (to be published).

\bibitem[Tartakovsky and Brown(2008)]{Tartakovsky&Brown-IEEEAES08}
A. G. Tartakovsky and J. Brown, 
\newblock Adaptive spatial-temporal filtering methods for clutter removal and target tracking,
\newblock {\em IEEE Transactions on Aerospace and Electronic Systems} {\bf 44} (2008), 1522--1537.

\bibitem[Verdenskaya and Tartakovskii(1991)]{Verd}
N. V. Verdenskaya and A. G. Tartakovsky,
\newblock Asymptotically optimal sequential testing of multiple hypotheses for nonhomogeneous Gaussian processes in an asymmetric situation,
\newblock {\em Theory of Probability and its Applications}  {\bf 36} (1991), 536--547.
  
  \bibitem[Wald(1945)]{wald45}
A. Wald, 
\newblock Sequential tests of statistical hypotheses.
\newblock {\em Annals of Mathematical Statistics}, {\bf 16} (1945), 117--186.

\bibitem[Wald(1947)]{wald47}
A. Wald, 
\newblock {\em Sequential Analysis}, John Wiley \& Sons, Inc. New York, USA, 1947.

\bibitem[Wald and Wolfowitz(1948)]{wald48}
A. Wald and J.  Wolfowitz, 
\newblock Optimum character of the sequential probability ratio test,
\newblock {\em Annals of Mathematical Statistics}  {\bf 19} (1948), 326--339.

\bibitem{wijsman-AMS71}
R. A. Wijsman,
\newblock Exponentially bounded stopping time of sequential probability ratio tests for composite hypotheses,
\newblock {\em Annals of Mathematical Statistics} {\bf 42} (1971), 1859--1869.

\bibitem[Wong(1968)]{WongAMS-1968}
S. P. Wong,
\newblock Asymptotically optimum properties of certain sequential tests,
\newblock \emph{Annals of Mathematical Statistics} {\bf 39} (1968), 1244--1263.
  



\end{thebibliography}
\end{document}